\documentclass[3p]{elsarticle}
\usepackage{amsmath,amssymb,amsfonts,amsthm,amsopn}
\usepackage{graphics}
\usepackage{xcolor}
\usepackage{pb-diagram}
\usepackage[title]{appendix}
\usepackage{graphicx}
\usepackage{mathtools}
\usepackage{algorithm}      
\usepackage{algpseudocode}  
\usepackage{float}

\newtheorem{theorem}{Theorem}[section]
\newtheorem{lemma}[theorem]{Lemma}

\newtheorem{proposition}[theorem]{Proposition}

\newtheorem{remark}[theorem]{Remark}

\newcommand{\beqa}{\begin{eqnarray*}}
\newcommand{\eeqa}{\end{eqnarray*}}

\newcommand{\field}[1]{\mathbb{#1}}
\newcommand{\bR}{\field{R}}        
\newcommand{\bC}{\field{C}}        
\def\la{\lambda}

\def\<{\left<}
\def\>{\right>}

\def\mv1{M_v^1}

\def\mn{(m,n)}
\def\mn'{(m',n')}

\def\Ren{\mathbb{R}^d}

\def\Sn2{S_{2}(L^{2}(\Ren))}
\def\S1{S_{1}(L^{2}(\Ren))}
\def\sig00{\sigma_{0,0}}

\def\la{\langle}
\def\ra{\rangle}

\begin{document}

\begin{frontmatter}

\title{Beyond Single-Window Graph Fourier Analysis}

\author{Iulia Martina Bulai}
\address{Universit\`a di Torino, Dipartimento di Matematica, via Carlo Alberto 10, 10123 Torino, Italy, iuliamartina.bulai@unito.it}

\author{Elena Cordero}
\address{Universit\`a di Torino, Dipartimento di Matematica, via Carlo Alberto 10, 10123 Torino, Italy, elena.cordero@unito.it}
\author{Edoardo Pucci}
\address{Universit\`a di Torino, Dipartimento di Matematica, via Carlo Alberto 10, 10123 Torino, Italy, edoardo.pucci@unito.it }
\author{Sandra Saliani}
\address{Universit\`a di Napoli Parthenope, Dipartimento di Ingegneria, Centro Direzionale Isola C4, 80143 Napoli, Italy, sandra.saliani@uniparthenope.it}



\begin{abstract}
We introduce a multi-windowed graph Fourier transform (MWGFT) for the joint vertex-frequency analysis of signals defined on graphs. Building on generalized translation and modulation induced by the graph Laplacian, the proposed framework extends the windowed graph Fourier transform by allowing multiple analysis and synthesis windows. Exact reconstruction formulas are derived for complex-valued graph signals, together with sufficient and computable conditions guaranteeing stable invertibility. The associated families of windowed graph Fourier atoms are shown to form frames for the space
of graph signals. Numerical experiments on synthetic and real-world graphs confirm exact reconstruction up to machine precision and demonstrate improved stability and vertex–frequency localization compared to single-window constructions, particularly on irregular graph topologies.
\end{abstract}


\begin{keyword}
Graph Signal Processing; Vertex-frequency analysis;
 Generalized translation and modulation;
    Windowed Graph Fourier Transform
    \MSC[2020]{42A38,05C50}
\end{keyword}

\end{frontmatter}

\section{Introduction}
Graph signal processing (GSP) has emerged as a powerful framework for extending classical harmonic analysis to data supported on graphs and networks. By introducing graph-adapted analogues of fundamental tools such as the graph Fourier transform, \cite{9144464}, \cite{7937851}, \cite{osti_10092281}, \cite{ZHENG202135}, and its generalization \cite{ALIKASIFOGLU2025109944}, graph wavelets, \cite{HAMMOND2011129}, \cite{7355305}, graph wavelet packets, \cite{BULAI202318}, \cite{BREMER200695}, \cite{Dam2024}, \cite{song2026waveletpacketcontentpositiveoperators}, and other spectral filtering techniques, \cite{8347162}, \cite{10962535}, \cite{Saito2024}, \cite{8038007}, \cite{8907446}, GSP
provides representations that exploit the intrinsic geometry of the underlying
graph. These methods, which are based on the spectral properties of graph-related
operators such as the Laplacian, have found successful applications in denoising,
clustering, learning on networks, and the identification of structural patterns
in complex, high-dimensional data, For some examples see for instance \cite{10.1371/journal.pone.0261041}, \cite{10.3389/fnins.2019.00585}, \cite{10.1371/journal.pone.0128136}, \cite{8307490}, \cite{8309033}, \cite{Behjat2023VoxelWise}.

Related but complementary approaches in graph signal processing focus on
approximation and interpolation, including kernel-based and partition-of-unity
methods \cite{ERB2022368}, \cite{CAVORETTO2024128502}, \cite{Cavoretto20246685}, \cite{Cavoretto202225}.
While these techniques emphasize stable reconstruction and smoothing, they do
not provide an explicit vertex--frequency representation. Other distinct developments include cross-spectral analysis of graph signals
\cite{Kim2024CrossSpectralAO}, and
methods for dynamic graphs \cite{GENG2023101579}, with applications such as neuroimaging further highlighting the need for joint vertex--frequency analysis
\cite{10095285}.

Numerous contributions have addressed the construction of localized analysis tools on graphs. Among these, graph wavelet transforms represent one of
the most widely studied approaches. Wavelet-based methods typically rely on
spectral filtering of the graph Laplacian, where a single generating kernel is
dilated across scales to produce families of localized atoms. These
constructions are primarily designed to capture multiscale behavior and have
proven effective in applications involving hierarchical or scale-dependent
signal features; see \cite{BREMER200695}, \cite{BULAI202318}, \cite{COIFMAN200653}, \cite{Dam2024}, \cite{HAMMOND2011129}
and references therein. For recent overviews of spectral wavelet constructions
and related multiscale tools we refer the reader to \cite{Fonseca2024}.

In parallel, a different line of research focuses on vertex--frequency analysis,
aiming to generalize classical time--frequency methods to graph domains. A
fundamental contribution in this direction is the windowed graph Fourier
transform introduced by Shuman, Ricaud, and Vandergheynst in \cite{Shuman2016}.
Their framework establishes generalized notions of translation and modulation
on graphs, defined intrinsically through the graph Laplacian spectrum, and
leads to a graph Fourier transform analogue of the short-time Fourier transform
(STFT). This approach provides a joint vertex--frequency representation that
directly mirrors classical Gabor analysis and enables the study of localized
oscillatory phenomena on graphs. Related frame-based and Gabor-type
constructions have been explored as well; in particular, Ghandehari et al.
investigated Gabor-type frames for graph signal processing \cite{Ghandehari2021},
and more recent works consider vertex--frequency representations in generalized
spectral domains such as fractional graph Fourier transforms \cite{SZ2026}.

The present work builds directly on the vertex--frequency framework of
\cite{Shuman2016}, extending it in a direction inspired by multi-window Gabor
analysis in the Euclidean setting. In classical signal processing, the
limitations of single-window Gabor representations are well known and have
motivated the introduction of multi-window schemes, which combine several
windows with different localization properties. Such constructions were
originally rooted in Gabor’s theory of communication \cite{Gabor1946} and later
developed into a robust frame-theoretic setting with important contributions by
Zeevi, Zibulski, and Porat \cite{Zeevi1998}, \cite{ZibulskiZeevi1997}. These works show
that multi-window Gabor systems provide enhanced flexibility, improved numerical
stability, and better adaptability to heterogeneous signal structures.

Unlike graph wavelet approaches, which are inherently based on spectral scaling
(dilations) of a single kernel, the methodology developed in this paper is
fundamentally Gabor-based. The proposed multi-windowed graph Fourier transform
(MWGFT) relies on the independent use of generalized translation and modulation
operators, allowing windows to be shifted across vertices and modulated across
graph frequencies. As a
result, the MWGFT yields an explicit vertex–frequency representation, in which
spatial localization and frequency content are indexed separately and can be
controlled independently.

Signals supported on graphs naturally arise in a wide range of applications,
including sensor and communication networks, transportation and power grids,
social and biological networks, and data defined on irregular geometric
structures. In these settings, signals are indexed by the vertices of a graph
rather than by points on a regular domain, and their structure is intrinsically
influenced by the topology of the underlying network. This intrinsic
irregularity prevents the direct application of classical signal processing
tools and motivates the development of analysis methods that are adapted to
graph domains; classical references and overviews on networks and spectral graph
theory remain central \cite{Newman2010Networks}, \cite{chung1997spectral}.

Graph signal processing (GSP) provides a general framework for extending
harmonic analysis to signals defined on graphs. A central role in GSP is played
by the graph Laplacian, whose spectral decomposition yields a notion of
frequency that reflects the connectivity of the graph. Given a finite,
connected, undirected weighted graph $G=(V,E)$ with $|V|=N$, the (unnormalized)
graph Laplacian is defined by
\[
L = D - W,
\]
where $W$ is the weighted adjacency matrix and $D$ is the degree matrix (see
Section~\ref{sec:definitions} for notation). Since $L$ is real and symmetric,
it admits an orthonormal eigenbasis $\{\chi_\ell\}_{\ell=0}^{N-1}$ associated
with nonnegative eigenvalues $0=\lambda_0<\lambda_1\le\cdots\le\lambda_{N-1}$.
This spectral structure induces the graph Fourier transform (GFT) of a signal
$f\in\mathbb{C}^N$, defined by $\widehat f(\ell)=\langle f,\chi_\ell\rangle$,
with inversion $f(i)=\sum_{\ell}\widehat f(\ell)\chi_\ell(i)$.

While the GFT provides global spectral information, many real-world graph
signals exhibit strong spatial localization and nonstationary behavior.
Purely spectral representations are therefore often insufficient for practical
analysis, which motivates the construction of joint vertex--frequency
representations (cf.\ \cite{Shuman2016} and references therein).

In the WGFT framework, generalized translation and modulation are defined by
the Laplacian eigenstructure. For $k\in\{0,\dots,N-1\}$ and a window
$g\in\mathbb{C}^N$ one sets
\[
(M_k f)(i):=\sqrt{N}\,f(i)\chi_k(i),\qquad
(T_n f)(i):=\sqrt{N}(f*\delta_n)(i)=\sqrt{N}\sum_{\ell=0}^{N-1}\widehat f(\ell)\chi_\ell^*(n)\chi_\ell(i),
\]
and the windowed atoms are $g_{n,k}:=M_k T_n g$, with WGFT coefficients
$S_f(n,k)=\langle f,g_{n,k}\rangle$. This yields a joint vertex--frequency
representation closely analogous to the classical short-time Fourier transform.

Our first set of contributions revisits this WGFT construction and develops a
two-window reconstruction theory where analysis and synthesis windows may
differ. Concretely, given an analysis window $g$ and a synthesis window
$\gamma$, we derive an explicit reproducing formula and show that, under the
mild non-degeneracy condition $\langle T_n\gamma,T_n g\rangle\neq 0$ for all
vertices $n$, the reconstruction holds exactly (see Theorem~\ref{RecTeo1}).
Moreover, the families $\{g_{n,k}\}$ and $\{\gamma_{n,k}\}$ are shown to form
frames for $\mathbb{C}^N$ (Theorem~\ref{impl1}), yielding a frame-theoretic
interpretation of the WGFT even when analysis and synthesis windows differ.

Building on these two-window results, the main methodological advance of the
paper is the formalization and analysis of the multi-windowed graph Fourier
transform (MWGFT). We consider collections of analysis and synthesis windows
$g_1,\dots,g_J$ and $\gamma_1,\dots,\gamma_J$, form the corresponding
atoms $g^j_{n,k}=M_k T_n g_j$ and $\gamma^j_{n,k}=M_k T_n\gamma_j$, and prove
an exact multi-window reconstruction formula. Under the vertex-wise
non-degeneracy condition
\[
\sum_{j=1}^J \langle T_i\gamma^j,T_i g^j\rangle\neq 0,\qquad i=1,\dots,N,
\]
any $f\in\mathbb{C}^N$ can be recovered from its MWGFT coefficients via a
weighted superposition of the synthesis atoms (Theorem~\ref{Recteo2}). These
results are complemented by spectral-domain sufficient conditions that are
computable in practice and guarantee stable invertibility and numerically
well-conditioned reconstructions. The theory further shows that the multi-window
systems constitute frames for $\mathbb{C}^N$, thereby connecting the MWGFT to
classical multi-window Gabor frame theory \cite{Gabor1946}, \cite{Zeevi1998}, \cite{ZibulskiZeevi1997}
and to recent graph-based Gabor constructions \cite{Ghandehari2021}.

The practical advantages of the MWGFT are illustrated through numerical
experiments on both regular and irregular graphs. Using spectral-window design
based on shifted radial-basis kernels and exact Laplacian eigendecompositions,
we demonstrate exact reconstruction up to machine precision, verify the
non-degeneracy conditions numerically, and show improved vertex--frequency
localization when multiple windows are employed. The present implementation
relies on full eigendecomposition and is therefore suited to graphs of moderate
size; scalable variants based on approximate spectral filtering (e.g., Chebyshev
polynomials) are discussed as future work.

To the best of our knowledge, this paper is the first to provide a comprehensive analysis and exact reconstruction theory for WGFT-type vertex--frequency
representations that allows \emph{different} analysis and synthesis windows and
their rigorous combination in a multi-window setting. By bringing together
multi-window Gabor theory and graph Fourier analysis, we offer a flexible,
frame-theoretic approach for localized spectral analysis of graph signals.

The remainder of the paper is organized as follows. Section~\ref{sec:definitions} reviews basic
definitions and notation in graph signal processing. Section~\ref{sec:WFT_rec_formula} introduces the
two-window WGFT framework and derives reconstruction and frame results. The
multi-window extension is developed in Section~\ref{sec:MWGFT}. Numerical experiments
illustrating exact reconstruction, stability, and vertex--frequency localization
on both regular and irregular graphs are presented in Section~5. Finally,
Section~6 concludes the paper and outlines directions for future work.

\section{Definitions and notations}\label{sec:definitions}

We consider a weighted simple graph $G=(V,E)$ which is connected, undirected and finite, where $V=\{1,\dots,N\}$ is the set of vertices and $E$ the edges connecting the nodes.
  We denote by $W=(W_{ij})_{i,j=1,\dots,N}$  the real \textit{weighted adjacency matrix}, i.e., $W_{ij}\geq 0$ represents the weight of the edge from the vertex $i$ to the vertex $j$, $W_{ij}=0$ means that $i$ and $j$ are not connected by an edge. Since we deal with undirected simple graphs, $W$ is symmetric with all zeros on the diagonal.
 If $W_{ij}\ne 0$ then we say that $i$ and $j$ are adjacent and we write $i\sim j$.

A complex \textit{signal} on $G$ is a function $f:G\to \mathbb{C}$ or equivalently a vector \\ $f=[f(1),\dots,f(N)]^T\in\mathbb{C}^N$. \\The \textit{degree matrix} of $G$ is given by $D=\operatorname{diag}(d_1,\dots,d_{N})$ with $$d_i:=\displaystyle\sum_{j\sim i}W_{ij}$$ being the degree of the vertex $i$.
The real symmetric matrix $$\mathcal{L}:=D-W$$ is called the unnormalized \textit{Laplacian matrix} of $G$.
The action of $\mathcal{L}$ on a signal $f\in\mathbb{C}^N$ is given by:\begin{align*}
    (\mathcal{L}f)(i)&=(Df)(i)-(Wf)(i)
    =\ d_if(i)-\displaystyle\sum_{j=1}^{N} W_{ij}f(j)
    = \displaystyle\sum_{j\sim i}W_{ij}f(i)-\displaystyle\sum_{j\sim i} W_{ij}f(j)\\
    &=\displaystyle\sum_{j\sim i}W_{ij}[f(i)-f(j)],
\end{align*}
for every $i=1,\dots, N$. Since $\mathcal{L}$ is real and symmetric, there exists an orthonormal basis $(\chi_0,\dots,\chi_{N-1})$ of $\mathbb{C}^N$ made up of eigenvectors of $\mathcal{L}$ and its eigenvalues are real.

\begin{proposition}
    The eigenvalues of $\mathcal{L}$ are non-negative, in particular $\lambda=0$ is an eigenvalue and the corresponding eigenspace has dimension 1.
\end{proposition}

\begin{proof}
    If $\lambda$ is an eigenvalue of $\mathcal{L}$, then there exists a unit associated eigenvector $\phi$. Hence,
    \begin{equation*}
        \lambda=\la \mathcal{L}\phi,\phi\ra=\displaystyle\sum_{j\sim i}W_{ij}[\phi(i)-\phi(j)]^2\geq 0.
    \end{equation*}
    Since we assumed that $G$ is connected,  $f\in\mathbb{C}^N\setminus\{0\}$ solves $\mathcal{L}f=0$ if and only if $f(i)=f(j),$ for every  $i,j=1,\dots, N$. Thus, $\lambda=0$ is an eigenvalue and the associated eigenspace is given by $\{f=c(1,\dots,1)^T|\ c\in\mathbb{C}\}$.
\end{proof}

We fix an orthonormal eigenbasis of $\mathcal{L}:$ $$(\chi_0,\dots,\chi_{N-1}),$$
 adopting the convention that the related  eigenvalues are in ascending order:
 $$0=\lambda_0< \lambda_1\leq \lambda_2\leq\dots\leq\lambda_{N-1}.$$
  In this way, $\chi_0$ is the (only) eigenvector of the basis which is associated to the eigenvalue $\lambda_0=0$ and since $\|\chi_0\|_2=1$ we have that \begin{equation}
     |\chi_0(i)|=\frac{1}{\sqrt{N}},\quad \forall i=1,\dots,N.
 \end{equation}

The definitions and notations below are inferred from \cite{Shuman2016}.

    The graph Fourier transform (GFT) $\hat{f}$ of a signal $f \in \mathbb{C}^N$ is the expansion of $f$ in terms of the fixed eigenvectors of the graph Laplacian. It is defined by
\[
\hat{f}(\ell) := \langle f, \chi_\ell \rangle = \sum_{i=1}^N f(i) \chi_\ell^*(i),\quad \forall \ell=0,\dots,N-1,
\]
where we adopt the convention that the inner product is conjugate-linear in the second argument. We have the following expression for $|\hat f(0)|$:\begin{equation}\label{GFT0}
    |\hat f(0)|=\bigg|\displaystyle\sum_{i=1}^{N}f(i)\chi_0(i)\bigg|=|\chi_0(0)|\bigg|\displaystyle\sum_{i=1}^{N} f(i)\bigg|=\frac{1}{\sqrt{N}}\bigg|\displaystyle\sum_{i=1}^{N} f(i)\bigg|.
\end{equation}The inverse graph Fourier transform is given by
\[
f(n) = \sum_{\ell=0}^{N-1} \hat{f}(\ell) \chi_\ell(n).
\]

We denote the largest modulus of the elements of a given graph Laplacian eigenvector by
\[
\mu_\ell := \|\chi_\ell\|_\infty = \max_{i \in \{1, 2, \dots, N\}} |\chi_\ell(i)|,
\]
and the largest modulus of a given row of $\chi$ by
\begin{equation}\label{nui}
    \nu_n := \max_{\ell \in \{0, 1, \dots, N-1\}} |\chi_\ell(n)|.
\end{equation}
Note that
\begin{equation}\label{ele2}
    \mu = \max_{\ell \in \{0, 1, \dots, N-1\}} \{\mu_\ell\} = \max_{i \in \{1, 2, \dots, N\}} \{\nu_i\}.
\end{equation}

For any $k \in \{0, 1, \dots, N-1\}$, the \emph{generalized modulation} operator $M_k: \mathbb{C}^N \to \mathbb{C}^N$ is defined by
\begin{equation}\label{mod}
(M_k f)(i) := \sqrt{N} f(i) \chi_k(i),\quad i=1,\dots,N.
\end{equation}
The \emph{generalized convolution} of signals \( f, g \in \mathbb{C}^N \) on a graph is given by
\begin{equation}
    (f \ast g)(i) := \sum_{\ell=0}^{N-1} \hat{f}(\ell) \hat{g}(\ell) \chi_\ell(i).
    \label{eq:generalized_convolution}
\end{equation}

Equivalently, denoting by $\chi$ the matrix whose $\ell$-th column is $\chi_\ell$, we can write the generalized convolution as
\begin{equation}
    f \ast g = \chi
    \begin{bmatrix}
        \hat{g}(0) & 0 & \cdots & 0 \\
        0 & \hat{g}(1) & \cdots & 0 \\
        \vdots & \vdots & \ddots & \vdots \\
        0 & 0 & \cdots & \hat{g}({N-1})
    \end{bmatrix}
    \chi^* f.
    \label{eq:generalized_convolution_matrix}
\end{equation}

For any signal $f \in \mathbb{C}^N$ defined on the graph $G$ and any $n\in \{1, 2, \dots, N\}$, we  define a \emph{generalized translation} operator $T_n: \mathbb{C}^N \to \mathbb{C}^N$ via generalized convolution with a delta centered at vertex $n$:
\begin{equation}\label{Ti}
    (T_n f)(i) := \sqrt{N} (f \ast \delta_n)(i) = \sqrt{N} \sum_{\ell=0}^{N-1} \hat{f}(\ell) \chi_\ell^*(n) \chi_\ell(i).
\end{equation}
Note that $\hat{\delta}_n(\ell)= \chi_\ell^*(n)$, for $\ell=0,\dots{N-1}$.

\section{Windowed graph Fourier  transform, reconstruction formula and frame property}\label{sec:WFT_rec_formula}
In what follows we recall the definition of windowed graph Fourier transform (WGFT) and generalize the reconstruction formula in \cite{Shuman2016} to the case of two different windows, satisfying suitable assumptions.
Finally, we study the frame property of  the related WGFT atoms.

Namely, for a window $g \in \mathbb{C}^N$, we define the windowed graph Fourier atom related to $g$ by
\begin{equation}\label{e2}
g_{n,k}(i) := (M_k T_n g)(i) = N \chi_k(i) \sum_{\ell=0}^{N-1} \hat{g}(\ell) \chi_\ell^*(n) \chi_\ell(i),
\end{equation}
for $i,n\in\{1,\dots,N\}$, $k\in\{0,\dots,N-1\} $, and the windowed graph Fourier transform of a function $f \in \mathbb{C}^N$ associated to \eqref{e2} by
\[
S_f(n, k) := \langle f, g_{n,k} \rangle,\quad n\in \{1,\dots,N\},\,\,k\in\{0,\dots,N-1\}.
\]
The next result generalizes  
Lemma $1$ in \cite{Shuman2016}.
\begin{lemma}\label{upperest1}
    For any $g,\gamma \in \mathbb{C}^N$, $n=1,\dots,N$,
    \[
    |\la T_ng,T_n\gamma\ra|  \leq {N} \nu_n^2 \|g\|_2\|\gamma\|_2 \leq {N} \mu^2 \|g\|_2\|\gamma\|_2,
    \]
    where $\nu_n$ is defined in \eqref{nui}.
\end{lemma}
\begin{proof}
    By definition of translation in \eqref{Ti},
    \begin{align}
        \la T_ng,T_n\gamma\ra &= \sum_{i=1}^NT_ng(i)\overline{T_n\gamma(i)}\notag\\
        &    = \sum_{i=1}^N  {N} \sum_{\ell=0}^{N-1} \hat{g}(\ell) \chi_\ell^*(n) \chi_\ell(i)\overline{ \sum_{\ell'=0}^{N-1} \hat{\gamma}(\ell') \chi_{\ell'}^*(n) \chi_{\ell'}(i) }\notag\\
        &= N \sum_{\ell=0}^{N-1} \sum_{\ell'=0}^{N-1} \hat{g}(\ell) \overline{\hat{\gamma}(\ell')} \chi_\ell^*(n) \chi_{\ell'}(n) \sum_{i=1}^N \chi_\ell(i) \chi_{\ell'}^*(i)\notag\\
            &= N \sum_{\ell=0}^{N-1} \sum_{\ell'=0}^{N-1} \hat{g}(\ell) \overline{\hat{\gamma}(\ell')} \chi_\ell^*(n) \chi_{\ell'}(n) \delta_{\ell \ell'}\notag\\
        &= N \sum_{\ell=0}^{N-1} \hat{g}(\ell)\overline{\hat{\gamma}(\ell)}  |\chi_\ell^*(n)|^2. \label{29}
    \end{align}
    Hence, using Cauchy-Schwarz inequality and \eqref{nui} and \eqref{ele2},
    \begin{align*}
        |\la T_ng,T_n\gamma\ra|
        &\leq  N \left(\sum_{\ell=0}^{N-1} |\hat{g}(\ell)\chi_\ell(n)|^2\right)^{1/2}\left(\sum_{\ell=0}^{N-1} |\hat{\gamma}(\ell)\chi_\ell(n)|^2\right)^{1/2}\\
        &\leq N \nu_n^2 \|\hat{g}\|_2^2\|\hat{\gamma}\|_2^2   =N \nu_n^2 \|{g}\|_2^2\|{\gamma}\|_2^2,
    \end{align*}
    which yields the desired result.
\end{proof}
\begin{remark}
   Observe that, choosing the same window  $g=\gamma\in\mathbb{C}^N$, we obtain \begin{equation}\label{Uestimate}
        \|T_ng\|^2_2\leq N\nu_n^2\|g\|_2^2\leq N\mu^2\|g\|_2^2,\quad\forall n=1,\dots,N.
    \end{equation}
\end{remark}
  Another important inequality is the lower bound given in Lemma 1 of \cite{Shuman2016}, which is here generalized to the case of complex-valued signals.
    \begin{lemma}
    For a signal $f\in\mathbb{C}^N$ we have:\begin{equation}\label{Lestimate}
        |\hat f(0)|\leq \|T_nf\|_2,\quad\forall n=1,\dots,N.
    \end{equation}
    \end{lemma}
    \begin{proof}
         We first observe that $T_n$ preserves the modulus of the mean of $f$:
         \begin{align*}
        \bigg|\displaystyle\sum_{i=1}^{N} T_nf(i)\bigg|=&\  \sqrt{N}|\widehat{T_nf}(0)|
        =\  N|\widehat{(f\ast \delta_n)}(0)|\\
        =\  &N|\hat f(0)| \,|\hat{\delta}_n(1)|
        =\  N|\hat f(0)|\,|\chi_0(n)|\\
        =\ & N \frac{1}{\sqrt{N}}\bigg|\displaystyle\sum_{i=1}^{N} f(i)\bigg| \frac{1}{\sqrt{N}}
        =\  \bigg|\displaystyle\sum_{n=1}^{N} f(j)\bigg|.
    \end{align*}
    So\begin{align*}
        |\hat f(0)|=\frac{1}{\sqrt{N}}\bigg|\displaystyle\sum_{i=1}^{N} T_nf(i)\bigg|\leq \frac{1}{\sqrt{N}}\|T_nf\|_1\leq \|T_nf\|_2.
    \end{align*}
    In the last inequality we used  the equivalence of  norms on $\mathbb{C}^N$:\begin{equation*}
        \|x\|_2\leq \|x\|_1\leq \sqrt{N} \|x\|_2,\qquad \forall x\in\mathbb{C}^N.
    \end{equation*}
   This concludes the proof. \end{proof}

In what follows we generalize the reconstruction formula in Theorem $4$ of \cite{Shuman2016} to the case of two windows $g,\gamma\in  \mathbb{C}^N$, provided the  windows satisfy
\begin{equation}\label{e3}
    \la T_n\gamma, T_ng\ra\not=0,\quad n\in\{1,\,\dots,N\}.
\end{equation}

 A signal $f \in  \mathbb{C}^N$ can be recovered from its windowed graph Fourier transform coefficients, as shown below.
\begin{theorem}\label{RecTeo1}
Fix two windows function $g,\gamma\in\bC^N$ satisfying condition \eqref{e3}. Then for any $f \in \mathbb{C}^N$,
\begin{equation}\label{recformula1}
f(i) = \frac{1}{N    \la T_i\gamma, T_ig\ra} \sum_{n=1}^N \sum_{k=0}^{N-1} S_f(n, k) \gamma_{n,k}(i),\quad \forall i=1,\dots N.
\end{equation}
\end{theorem}
\begin{proof}
    We have
\begin{align*}
\sum_{n=1}^N \sum_{k=0}^{N-1} & S_f(n, k) \gamma_{n,k}(i) \\
    &=N^2 \sum_{n=1}^N \sum_{k=0}^{N-1} \Bigg[ \sum_{m=1}^N f(m) \overline{\chi_k(m) \sum_{\ell=0}^{N-1} \hat{g}(\ell) \chi_\ell^\ast(n) \chi_\ell(m) }\Bigg]  \chi_k(i) \sum_{\ell'=0}^{N-1} \hat{\gamma}(\ell') \chi_{\ell'}^\ast (n)\chi_{\ell'}(i)  \\
    &=N^2  \sum_{m=1}^N f(m) \sum_{\ell,\ell'=0}^{N-1} \sum_{k=0}^{N-1} \overline{\hat{g}}(\ell) \hat{\gamma}(\ell') \chi_k^*(m)\chi_\ell^*(m) \chi_k(i)\chi_{\ell'}(i) \sum_{n=1}^N \chi_\ell(n) \chi_{\ell'}^*(n) \\
    &= N^2 \sum_{m=1}^N f(m) \sum_{\ell,\ell'=0}^{N-1} \sum_{k=0}^{N-1} \overline{\hat{g}}(\ell) \hat{\gamma}(\ell') \chi_k^*(m)\chi_\ell^*(m) \chi_k(i)\chi_{\ell'}(i)\delta_{\ell \ell'} \\
    &= N^2 \sum_{m=1}^N f(m) \sum_{\ell=0}^{N-1} \sum_{k=0}^{N-1} \overline{\hat{g}}(\ell) \hat{\gamma}(\ell) \chi_k^*(m)\chi_\ell^*(m) \chi_k(i)\chi_{\ell}(i)\\
        &= N^2  \sum_{m=1}^N f(m) \sum_{\ell=0}^{N-1}  \overline{\hat{g}}(\ell) \hat{\gamma}(\ell) \chi_\ell^*(m)\chi_{\ell}(i)\sum_{k=0}^{N-1}\chi_k^*(m) \chi_k(i)\\
        &=N^2  \sum_{m=1}^N f(m) \sum_{\ell=0}^{N-1}  \overline{\hat{g}}(\ell) \hat{\gamma}(\ell) \chi_\ell^*(m)\chi_{\ell}(i)\delta_{mi}\\
            &= N^2 f(i) \sum_{\ell=0}^{N-1}  \overline{\hat{g}}(\ell) \hat{\gamma}(\ell) \chi_\ell^*(i)\chi_{\ell}(i)\\
        &= N f(i)\la T_i\gamma, T_ig\ra,
\end{align*}
where the last equality follows from \eqref{29}.
\end{proof}
\begin{remark}
    The single window case $\gamma=g$ in $\bR^N$ recaptures the result in \cite[Theorem $4$]{Shuman2016}.
\end{remark}
\begin{remark}
  To simplify the computation of the reconstruction formula, one may impose the condition \begin{equation}\label{traslcost}
        |\la T_ng,T_n\gamma\ra|=1,\quad \forall n=1,\dots,N.
    \end{equation}
  Recall  the equality in \eqref{29}:\begin{equation*}
        \la T_ng,T_n\gamma\ra=N\displaystyle\sum_{\ell=0}^{N-1}\hat g(\ell)\overline{\hat \gamma(\ell)}|\chi_{\ell}^*(n)|^2,\quad \forall n=1,\dots N.
    \end{equation*}
    Thus, for example, one way to obtain condition \eqref{traslcost} is to choose $g=\gamma$ such that \begin{equation*}
        |\hat g(\ell)|=\frac{1}{\sqrt{N}},\quad \forall \ell=0,\dots N-1.
    \end{equation*}
\end{remark}
\subsection{Sufficient conditions for the reconstruction}

The reconstruction formula with two windows $g,\gamma\in\mathbb{C}^N$ requires the condition \eqref{e3}. In particular, if $g=\gamma$,  it reads as
    \begin{equation}\label{C2}
        \|T_ng\|_2\ne 0,\quad \forall n=1,\dots,N.
    \end{equation}
    In this case, the estimate $
        |\hat g(0)|\leq \|T_n g\|_2$
    provides the {sufficient} condition \begin{equation}\label{C3}\hat g(0)\ne 0,
    \end{equation}
which ensures \eqref{C2}.

Condition \eqref{C3} is stronger than \eqref{C2}, yet it is more readily \emph{computable} since
\begin{itemize}
\item[(i)] If we generate the window $g$ directly in the spectral domain, i.e., we generate $\hat g(\cdot)$, there is only  condition \eqref{C3} to check.
\item[(ii)] If we generate $g$ in the vertex domain, checking \eqref{C3} could be compared with checking if $\|T_ng\|_2$ is non-zero for only one $n$. So, if $N$ is large (as in most applications), the difference between \eqref{C2} and \eqref{C3} becomes more evident.
\end{itemize}
Also note that the Cauchy-Schwarz inequality
$$ |\la T_n\gamma, T_n g\ra|\leq \|T_n\gamma\|_2\;\|T_n g\|_2,\quad n\in\{1,\,\dots,N\},$$
implies that the following are equivalent:
\begin{itemize}
\item[a)] $\|T_ng\|_2\neq 0$, for all $n=1,\dots,N;$
\item[b)] There exists $\gamma$ such that $\la T_n\gamma, T_n g\ra\neq 0$ for all $n=1,\dots,N.$
\end{itemize}
Now we would like to generalize \eqref{C3} to the case of two windows and find a condition which implies \eqref{e3} and which is independent of the index  $n$.
\begin{lemma}\label{lemmasuff1}
    Let $g,\gamma\in\bC^N$, if\begin{equation}\label{csuff1}
        \mathfrak{Re}(\overline{\hat g(\ell)}\hat \gamma(\ell))\geq 0,\quad \forall \ell=1,\dots,N-1,\quad\mbox{and}\quad \mathfrak{Re}(\overline{\hat g(0)}\hat \gamma(0))>0,
    \end{equation}
    then condition \eqref{e3} is satisfied.
\end{lemma}
\begin{proof}
    From the equality in  \eqref{29} we have:\begin{equation*}
        \mathfrak{Re}(\la T_n\gamma,T_ng\ra)=N\displaystyle\sum_{\ell=0}^{N-1}\mathfrak{Re}(\overline{\hat g(\ell)}\hat \gamma(\ell))|\chi_{\ell}^*(n)|^2.
    \end{equation*}
    Hence, if the condition \eqref{csuff1} holds, then\begin{equation*}
        \mathfrak{Re}(\la T_n\gamma,T_ng\ra)\geq N|\chi_0^*(n)|^2\mathfrak{Re}(\overline{\hat g(0)}\hat \gamma(0))=\mathfrak{Re}(\overline{\hat g(0)}\hat \gamma(0))>0.
    \end{equation*}
    This gives the desired result.
\end{proof}
\begin{remark} If the windows coincide $g=\gamma$ we obtain   condition  \eqref{C3}. Other sufficient conditions that can be obtained by simple modification in the previous proof  are:\begin{equation}\label{csuff1a}
     \mathfrak{Re}(\overline{\hat g(\ell)}\hat \gamma(\ell))\leq 0, \quad \forall \ell=1,\dots,N-1,\quad\mbox{and}\quad \mathfrak{Re}(\overline{\hat g(0)}\hat \gamma(0))<0;
    \end{equation}
   \begin{equation}\label{csuff1b}
        \mathfrak{Im}(\overline{\hat g(\ell)}\hat \gamma(\ell))\geq 0,\quad \forall \ell=1,\dots,N-1,\quad\mbox{and}\quad \mathfrak{Im}(\overline{\hat g(0)}\hat \gamma(0))>0;
    \end{equation}
    \begin{equation}\label{csuff1c}
      \mathfrak{Im}(\overline{\hat g(\ell)}\hat \gamma(\ell))\leq 0\quad \forall \ell=1,\dots,N-1,\quad\mbox{and}\quad \mathfrak{Im}(\overline{\hat g(0)}\hat \gamma(0))<0.
    \end{equation}
    These three conditions are trivially unsatisfied if $g=\gamma$. Of course, conditions  \eqref{csuff1b} and \eqref{csuff1c} do not  hold for real signals.
\end{remark}
\begin{lemma}\label{lemmasuff2}
    Given $g,\gamma\in\mathbb{C}^N$,  condition \begin{equation}\label{csuff2}
        |\hat g(0)+\hat\gamma(0)|>\sqrt{N}\mu\|\hat g-\hat \gamma\|_2,
    \end{equation}
   where $\mu$ is defined in \eqref{ele2}, implies that $\la T_ng,T_n\gamma\ra\ne0,\ \forall n=1,\dots,N$.
\end{lemma}
\begin{proof}
    Fix $n\in\{1,\dots,N\}$. By applying the {polarization identity} in $\mathbb{C}^N$, equipped with the standard Hermitian product, we have:\begin{align*}
        \la T_ng,T_n\gamma\ra = & \frac{1}{4}\bigg( \|T_ng+T_n\gamma\|^2_2-\|T_ng-T_n\gamma\|^2_2+ i\|T_ng+iT_n\gamma\|^2_2-i\|T_ng-iT_n\gamma\|^2_2\bigg)\\
        = & \frac{1}{4}\bigg( \|T_n(g+\gamma)\|^2_2-\|T_n(g-\gamma)\|^2_2+ i\big(\|T_ng+iT_n\gamma\|^2_2-\|T_ng-iT_n\gamma\|^2_2\big)\bigg).
    \end{align*}
    So \begin{align*}
        4\mathfrak{Re} (\la T_ng,T_n\gamma\ra)=\ & \|T_n(g+\gamma)\|^2_2-\|T_n(g-\gamma)\|^2_2   \geq |\hat g(0)+\hat\gamma(0)|^2-N\mu^2\|g-\gamma\|^2_2\\
        =\ & |\hat g(0)+\hat\gamma(0)|^2-N\mu^2\|\hat g-\hat \gamma\|^2_2.
    \end{align*}
   The minorization above is obtained using the inequalities \eqref{Uestimate} and \eqref{Lestimate}.
    Hence, if \eqref{csuff2} holds, then $\mathfrak{Re}\la T_ng,T_n\gamma\ra\ne 0$, so $\la T_ng,T_n\gamma\ra\ne 0$ for every $n=1,\dots, N$. This concludes the proof.
\end{proof}

\begin{remark}
  (i)  If $g=\gamma$, condition \eqref{csuff2} is equivalent to \eqref{C3}.\\
  (ii)  Roughly speaking, condition \eqref{csuff2}
  requires that the sum signal $g+\gamma$ is \emph{concentrated} along the direction of the first Laplacian eigenvector $\chi_0$ and an higher concentration is needed as the distance between $g$ and $\gamma$ increases.
\end{remark}

\noindent
\textbf{Frame Property}.
We first recall the definition of a frame in a separable Hilbert space.

Let $\mathcal{H}$ be a separable Hilbert space and $J$ an at most countable index set.

    We say that a sequence $\{f_n\}_{n \in J} \subset \mathcal{H}$ is a frame for $\mathcal{H}$ if there exist constants $A, B > 0$ such that
    \[
    A \|f\|^2 \leq \sum_{n \in J} |\langle f, f_n \rangle|^2 \leq B \|f\|^2,\quad \forall f\in  \mathcal{H}.
    \]
   The constants $A$ and $B$ are called the \emph{frame bounds}.
    A sequence $\{g_n\}_{n \in J} \subset \mathcal{H}$ is called a dual frame of the frame $\{f_n\}_{n \in J}$ if, for every $f \in \mathcal{H}$,
\begin{equation}\label{dualfr}
    f = \sum_{n \in J} \langle f, f_n \rangle g_n = \sum_{n \in J} \langle f, g_n \rangle f_n.
\end{equation}
    If the sequence $\{f_n\}_{n \in J}$ is a frame, an example of dual frame is the so-called \emph{canonical dual frame}
    \[
    g_n = S^{-1} f_n,
    \]
    where $S$ is the frame operator (positive and invertible) defined by
    \[
    Sf = \sum_{n \in J} \langle f, f_n \rangle f_n.
    \]
Note that, if $\mathcal{H}$ is
a finite-dimensional space, $\{f_n\}_{n \in J}$ is a frame if and only if
the vectors $\{f_n\}_{n \in J}$ will span $\mathcal{H}.$
Observe that \eqref{dualfr} means that the vectors $\{g_n\}_{n \in J}$ will span
 $\mathcal{H}$ whenever $\{f_n\}_{n \in J}$ spans. This implies that if \( \{f_n\}_{n \in J} \) is a frame for \( \mathcal{H} \), then any of its dual frames is also a frame for \( \mathcal{H} \).

 \begin{theorem}\label{impl1}
If  $g,\gamma\in\bC^N,$ verify condition \eqref{e3}, then
$$\{g_{n,k}\}_{n=1,\dots,N;k=0,\dots,N-1},\quad\text{and}\quad\{\gamma_{n,k}\}_{n=1,\dots,N;k=0,\dots,N-1}$$
are frames for \( \mathbb{C}^N \). Moreover, each frame is the dual of the other.
 \end{theorem}
\begin{proof} Condition \eqref{e3} implies
\begin{equation*}
	0\neq |\la T_i\gamma, T_i g\ra|\leq \|T_i\gamma\|_2\;\|T_i g\|_2,\quad i\in\{1,\,\dots,N\},
\end{equation*}
so that
\begin{equation}\label{causw}
	0<a:=\min_{i \in \{1, \dots, N\}} \frac{|\la T_i\gamma, T_i g\ra|}{\|T_i\gamma\|_2}\leq \|T_i g\|_2\leq \max_{i \in \{1, \dots, N\}}\|T_i g\|_2=:b.
\end{equation}
Let us now show the equality:
\begin{equation}\label{eq:25-corrected}
	\sum_{n=1}^{N}\sum_{k=0}^{N-1}\big|\langle f,g_{n,k}\rangle\big|^{2}
	= N\sum_{i=1}^{N}|f(i)|^{2}\,\|T_i g\|_{2}^{2}.
\end{equation}

	Using that the inner product is conjugate-linear in the second entry,
	\begin{align*}
		\langle f,g_{n,k}\rangle
		&=\sum_{i=1}^N f(i)\,\overline{g_{n,k}(i)} =\sum_{i=1}^N f(i)\,\overline{\sqrt{N}\,(T_ng)(i)\chi_k(i)} \\
		&=\sqrt{N}\sum_{i=1}^N f(i)\,\overline{(T_ng)(i)}\,\overline{\chi_k(i)}.
	\end{align*}
	For fixed \(n\), define
	\[
	h_n(i):=f(i)\,\overline{(T_ng)(i)},\qquad i=1,\dots,N.
	\]
	Then
	\[
	\langle f,g_{n,k}\rangle
	=\sqrt{N}\sum_{i=1}^N h_n(i)\,\overline{\chi_k(i)}
	=\sqrt{N}\,\langle h_n,\chi_k\rangle.
	\]
	Therefore,
	\begin{align*}
		\sum_{k=0}^{N-1}\big|\langle f,g_{n,k}\rangle\big|^2
		=\sum_{k=0}^{N-1}\Big|\sqrt{N}\,\langle h_n,\chi_k\rangle\Big|^2
		=N\sum_{k=0}^{N-1}\big|\langle h_n,\chi_k\rangle\big|^2 =N\|h_n\|_2^2,
	\end{align*}
	where the last step follows from Parseval's equality since \(\{\chi_k\}_{k=0}^{N-1}\) is an orthonormal basis.
		Next,
	\begin{align*}
		\|h_n\|_2^2
		&=\sum_{i=1}^N |h_n(i)|^2
		=\sum_{i=1}^N |f(i)|^2\,|(T_ng)(i)|^2.
	\end{align*}
	Hence,
	\[
	\sum_{k=0}^{N-1}\big|\langle f,g_{n,k}\rangle\big|^2
	=N\sum_{i=1}^N |f(i)|^2\,|(T_ng)(i)|^2.
	\]
	Summing over \(n=1,\dots,N\) and exchanging the finite sums gives
	\begin{align*}
		\sum_{n=1}^N\sum_{k=0}^{N-1}\big|\langle f,g_{n,k}\rangle\big|^2
		=\sum_{n=1}^N N\sum_{i=1}^N |f(i)|^2\,|(T_ng)(i)|^2
		=N\sum_{i=1}^N |f(i)|^2\sum_{n=1}^N |(T_ng)(i)|^2.
	\end{align*}
	
	It remains to show that, for each fixed \(i\),
	\[
	\sum_{n=1}^N |(T_ng)(i)|^2=\|T_ig\|_2^2.
	\]
	Define the \(N\times N\) matrix \(A\) by
	\[
	A_{i,n}:=(T_ng)(i).
	\]
	Then \(\sum_{n=1}^N |(T_ng)(i)|^2\) is the squared \(\ell^2\)-norm of the \(i\)-th row of \(A\),
	while \(\|T_ig\|_2^2=\sum_{m=1}^N |(T_ig)(m)|^2\) is the squared \(\ell^2\)-norm of the \(i\)-th column of \(A\).
	
	Using the translation formula
	\[
	(T_ng)(i)=\sqrt{N}\sum_{\ell=0}^{N-1}\widehat g(\ell)\,\overline{\chi_\ell(n)}\,\chi_\ell(i),
	\]
	we can write
	\[
	A=\sqrt{N}\,U\,\operatorname{diag}(\hat{g})
	\,U^*,
	\]
	where $U=[\chi_0,\ldots,\chi_{N-1}]$ denotes the orthonormal
	matrix of Laplacian eigenvectors, and $\operatorname{diag}(\hat{g})$ is the diagonal matrix having the entries of $\hat{g}$ on the principal diagonal.
	In particular, \(A\) is normal, so \(AA^*=A^*A\).
	Taking the \((i,i)\)-entry yields
	\[
	(AA^*)_{i,i}=(A^*A)_{i,i}.
	\]
	But
	\[
	(AA^*)_{i,i}=\sum_{n=1}^N |A_{i,n}|^2=\sum_{n=1}^N |(T_ng)(i)|^2,
	\qquad
	(A^*A)_{i,i}=\sum_{m=1}^N |A_{m,i}|^2=\sum_{m=1}^N |(T_ig)(m)|^2=\|T_ig\|_2^2.
	\]
	Hence \(\sum_{n=1}^N |(T_ng)(i)|^2=\|T_ig\|_2^2\), as claimed.
	
	Substituting this identity back gives
	\[
	\sum_{n=1}^N\sum_{k=0}^{N-1}\big|\langle f,g_{n,k}\rangle\big|^2
	= N\sum_{i=1}^N |f(i)|^2\,\|T_ig\|_2^2.
	\]
Finally, the estimates \eqref{causw} yield
	\[
a^2N\|f\|_2^2\leq \sum_{n=1}^N\sum_{k=0}^{N-1}\big|\langle f,g_{n,k}\rangle\big|^2\leq
b^2N\|f\|_2^2.
\]
Therefore,  the  sequence $\{g_{n,k}\}_{n=1,\dots,N;k=0,\dots,N-1}$ is a frame for $\bC^N,$ with frame bounds $A=a^2 N$ and $B=b^2N$.
The same argument with the window $\gamma$ in place of $g$ gives that the sequence $\{\gamma_{n,k}\}_{n=1,\dots,N;k=0,\dots,N-1}$
forms a frame as well, which is the dual frame by the reconstruction formula \eqref{recformula1}.
\end{proof}

\section{Multi-windows reconstruction}\label{sec:MWGFT}
    The original construction of Gabor \cite{Gabor1946} was motivated by the need for localized spectral analysis of nonstationary signals such as speech. His solution, an exponential kernel
    modulated by a Gaussian window, provided an effective representation for such signals. For other classes of nonstationary data, however, it may be advantageous to employ kernels different from the Gaussian-modulated exponential.
    A central limitation arises from the uncertainty principle: optimizing a single window
    necessarily restricts the achievable balance between time and frequency resolution.
    A natural way to mitigate this limitation is to generalize the scheme by employing multiple window functions simultaneously. For example, by combining both narrow and
    wide Gaussian windows one can achieve higher resolution in time and in frequency at once.
    More generally, a multiwindow approach allows the inclusion of windows of different
    functional types, chosen according to the local features of the signals under
    investigation.

    From a theoretical standpoint, the multiwindow Gabor construction leads to families
    of atoms whose completeness and frame properties can be analyzed in the Zak transform
    domain. The eigenvalues of the associated matrix-valued function control the frame
    bounds, making it possible to characterize tight frames and to identify dual frames
    explicitly. Thus, the use of multiple windows is motivated both by practical needs, flexible
    resolution and adaptability to diverse signal classes, and by theoretical considerations,
    since it provides a richer and more robust frame structure than the single-window case, cf. \cite{Zeevi1998}, \cite{ZibulskiZeevi1997}.

In Theorem \ref{RecTeo1} the signal is reconstructed through $$\{g_{n,k}\}_{n=1,\dots,N;k=0,\dots,N-1}\quad\mbox{and}\quad\{\gamma_{n,k}\}_{n=1,\dots,N;k=0,\dots,N-1},$$ which are families obtained by modulating and translating a single window (g in the first case and $\gamma$ in the second one).

A reconstruction formula can be found also when the families above are obtained starting from an arbitrary number $J\in\{1,2,\dots\}$ of windows. This multi-windowed reconstruction have been also derived in previous works using frame-theoretic arguments. Specifically, Zheng et al. \cite{ZHENG202135} provide reconstruction formulas based on the canonical dual of multi-windowed graph Fourier frames, under global spectral frame conditions. The next theorem presents an explicit reconstruction formula for systems generated by multiple analysis windows and distinct synthesis windows, relying on a local vertex-wise condition and avoiding the use of canonical dual frames or global tight frame assumptions.

\begin{theorem}\label{Recteo2}
    Let $J\in\{1,2,\dots\}$, $g^1,\dots,g^J,\gamma^1,\dots,\gamma^J\in\bC^N$. If \begin{equation}\label{e4}
        \displaystyle\sum_{j=1}^J \la T_i\gamma^j,T_ig^j\ra\ne 0,\quad \forall i=1,\dots,N,
    \end{equation}
    then, for any $f\in\bC^N$, \begin{equation}\label{recformula2}
        f(i)=\frac{1}{N \displaystyle\sum_{j=1}^J \la T_i\gamma^{j},T_ig^{j}\ra}\sum_{j=1}^J\sum_{n=1}^N \sum_{k=0}^{N-1} \la f,g^{j}_{n, k}\ra \gamma^{j}_{n,k}(i),\quad \forall i=1,\dots N.
    \end{equation}
\end{theorem}
\begin{proof}
From Theorem \ref{RecTeo1} we have that for every $j=1,\dots,J$:\begin{equation*}
    Nf(i) \la T_i\gamma^{j},T_ig^{j}\ra=\sum_{n=1}^N \sum_{k=0}^{N-1} \la f,g^{j}_{n, k}\ra \gamma^{j}_{n,k}(i).
\end{equation*}
By summing over $j$ we obtain:\begin{equation*}
    Nf(i)\displaystyle\sum_{j=1}^J \la T_i\gamma^{j},T_ig^{j}\ra=\sum_{j=1}^J\sum_{n=1}^N \sum_{k=0}^{N-1} \la f,g^{j}_{n, k}\ra \gamma^{j}_{n,k}(i),
\end{equation*}
as desired.
\end{proof}

Conditions \eqref{csuff1} and \eqref{csuff2} can be generalized to the multi-window case.

\begin{lemma}\label{lemmasuff3}
     Let $J\in\{1,2,\dots\}$, $g^1,\dots,g^J,\gamma^1,\dots,\gamma^J\in\bC^N$. The condition \begin{equation}\label{csuff3}
         \sum_{j=1}^J \mathfrak{Re}(\overline{\hat g^j(\ell)}\hat\gamma^j(\ell))\geq 0,\ \forall \ell=1,\dots,N-1,\quad\mbox{and}\quad \sum_{j=1}^J\mathfrak{Re}(\overline{\hat g^j(0)}\hat\gamma^j(0))>0,
     \end{equation}
     implies that
     \begin{equation}\label{multicond}
     \displaystyle\sum_{j=1}^J \la T_i\gamma^j,T_ig^j\ra\ne 0,\quad \forall i=1,\dots,N.
     \end{equation}
\end{lemma}
\begin{proof}
Fix an index $i\in\{1,\dots,N\}$. By following the same scheme of the proof of Lemma \ref{lemmasuff1} and summing over $j$:
    \begin{align*}
    \mathfrak{Re}\bigg(\displaystyle\sum_{j=1}^J \la T_i\gamma^j,T_ig^j\ra\bigg)=& N\mathfrak{Re}\bigg(\sum_{j=1}^J\sum_{\ell=0}^{N-1}\overline{\hat g^j(\ell)}\hat\gamma^j(\ell)|\chi^*_{\ell}(i)|^2\bigg)  =N\sum_{\ell=0}^{N-1}\bigg(\sum_{j=1}^J\mathfrak{Re}(\overline{\hat g^j(\ell)}\hat\gamma^j(\ell))\bigg)|\chi^*_{\ell}(n)|^2\\
    \geq & N\sum_{j=1}^J\mathfrak{Re}(\overline{\hat g^j(0)}\hat\gamma^j(0))|\chi^*_0(n)|^2
    = \sum_{j=1}^J\mathfrak{Re}(\overline{\hat g^j(0)}\hat\gamma^j(0))>0,
    \end{align*}
    and condition \eqref{multicond} follows.
\end{proof}

\begin{lemma}\label{lemmasuff4}
 Let $J\in\{1,2,\dots\}$, $g^1,\dots,g^J,\gamma^1,\dots,\gamma^J\in\bC^N$. If \begin{equation}\label{csuff4}
     \sum_{j=1}^J|\hat g^j(0)+\hat\gamma^j(0)|^2-N\mu^2\|\hat g^j-\hat\gamma^j\|_2^2>0,
 \end{equation}
 then  condition \eqref{multicond} is verified.
\end{lemma}
\begin{proof}
    As before, fix $i\in\{1,\dots,N\}$. From the proof of Lemma \ref{lemmasuff2} we know that:\begin{equation*}
        4\mathfrak{Re} (\la T_ig^j,T_i\gamma^j\ra)\geq |\hat g^j(0)+\hat\gamma^j(0)|^2-N\mu^2\|\hat g^j-\hat \gamma^j\|^2_2,\quad \forall j=1,\dots,J.
    \end{equation*}
    So, by summing over $j$ we immediately get the thesis.
\end{proof}

\begin{remark}
 A stronger sufficient condition can be deducted from \eqref{csuff4}. Namely,  if there exists $\ \bar j\in\{1,\dots,J\}\ s.t.$:\begin{equation}\label{csufff5}
     |\hat g^j(0)+\hat\gamma^j(0)|\geq\sqrt{N}\mu\|\hat g^j-\hat\gamma^j\|_2,\ \forall j\ne \bar j;\quad\mbox{and}\quad |\hat g^{\bar j}(0)+\hat\gamma^{\bar j}(0)|>\sqrt{N}\mu\|\hat g^{\bar j}-\hat\gamma^{\bar j}\|_2,
 \end{equation}
 then  \eqref{multicond} holds true.
\end{remark}
Finally, arguing as in Theorem \ref{impl1}, we obtain
 \begin{theorem}\label{implmulti}
  Let $J\in\{1,2,\dots\}$, $g^1,\dots,g^J,\gamma^1,\dots,\gamma^J\in\bC^N$. If they  verify condition \eqref{e4}, then
	$$\{g^j_{n,k}\}_{n=1,\dots,N;k=0,\dots,N-1,j=1,\dots,J},\quad\text{and}\quad\{\gamma^j_{n,k}\}_{n=1,\dots,N;k=0,\dots,N-1,j=1,\dots,J},$$
	are frames for \( \mathbb{C}^N \). Moreover, each frame is the dual of the other.
\end{theorem}

\section{Numerical simulations}

This section describes the numerical approach used to compute the
windowed graph Fourier transform (WGFT) and its multi-window extension
(MWGFT), and to validate the reconstruction formulas established in
Theorems~\ref{RecTeo1} and~\ref{Recteo2}. When $J=1$, the implementation reduces to the standard WGFT described in Section \ref{sec:WFT_rec_formula}, while for $J>1$ it yields the multi-window extension (MWGFT) analyzed in Section 4. The numerical experiments are designed to validate the theoretical reconstruction results, to illustrate vertex–frequency localization, and to highlight the advantages of multi-window representations over single-window constructions. All experiments are carried out
using the exact spectral decomposition of the graph Laplacian.

Recall that $G=(V,E)$ is a finite, connected, undirected weighted graph with
$|V|=N$. The normalized graph Laplacian $\mathcal{L}$ associated with $G$ is
constructed from the adjacency matrix, and its eigendecomposition
\[
\mathcal{L} = U \Lambda U^{*}
\]
is computed, where $U=[\chi_0,\ldots,\chi_{N-1}]$ denotes the orthonormal
matrix of Laplacian eigenvectors and
$\Lambda=\mathrm{diag}(\lambda_0,\ldots,\lambda_{N-1})$ the corresponding
eigenvalues in nondecreasing order. The eigenvectors provide the graph
Fourier basis used throughout the numerical implementation.

\subsection{Spectral window design and test signal generation}\label{sec:numercial_sim}
Throughout this section, $\widehat g_j$ and $\widehat\gamma_j$ denote spectral
windows (graph Fourier multipliers), while $g_j$ and $\gamma_j$ denote the
corresponding vertex-domain windows obtained via inverse graph Fourier transform.

Analysis and synthesis windows are specified in the spectral domain.
Given a window function $  \widehat g:[0,\lambda_{\max}]\rightarrow\mathbb{C}$,
the associated vertex--domain window is obtained as
\[
g = \sum_{\ell=0}^{N-1} \widehat g(\lambda_\ell)\chi_\ell.
\]
In the multi--window setting, several analysis windows
$\{\widehat g_j\}_{j=1}^J$ and synthesis windows
$\{\widehat \gamma_j\}_{j=1}^J$ are selected. These windows define the
graph filtering operators
\[
M_{g_j}=U\,\mathrm{diag}(\widehat g_j)\,U^{*},
\qquad
M_{\gamma_j}=U\,\mathrm{diag}(\widehat \gamma_j)\,U^{*},
\]
which are applied to Kronecker deltas to construct generalized translations on
the graph.

The MWGFT analysis is performed using a collection of spectral windows derived
from a shifted radial basis function (RBF) kernel. Specifically, we define a
prototype spectral window
\[
\widehat g(\lambda)
=
\exp\!\left(-\left(\frac{\lambda}{\ell_{\mathrm{fac}}\lambda_{\max}}\right)^4\right),
\]
where $\ell_{\mathrm{fac}}>0$ controls the localization of the corresponding
filters in the vertex domain. Given this reference spectral kernel, we construct
a family of $K$ analysis windows by spectral shifting,
\[
\widehat g_k(\lambda)=\widehat g(\lambda-\tau_k), \qquad k=1,\dots,K,
\]
where the shifts $\{\tau_k\}$ are chosen to ensure adequate coverage of the
normalized Laplacian spectrum $[0,2]$.

The spectral energy response associated with the family of windows is given by
\[
m(\lambda)=\sum_{k=1}^{K} |\widehat g_k(\lambda)|^2.
\]
Since the RBF kernel $\widehat g(\lambda)$ is strictly positive for all
$\lambda\in\mathbb{R}$, each shifted window $\widehat g_k(\lambda)$ is also
strictly positive. Consequently, $m(\lambda)>0$ for all
$\lambda\in[0,\lambda_{\max}]$, ensuring that the associated diagonal spectral
operator is invertible.

MWGFT analysis is performed by applying the spectral multipliers
$\widehat g_k(\lambda)$ in the graph Fourier domain, i.e., by filtering via the
operators $M_{g_k}$. For reconstruction, the
synthesis windows are obtained by normalizing the analysis windows with respect
to the spectral energy response, namely
\begin{equation}\label{eqn:gamma_synthesis}
\widehat\gamma_k(\lambda)=\frac{\widehat g_k(\lambda)}{m(\lambda)}.
\end{equation}
This normalization ensures that
$\sum_{k=1}^{K}\widehat\gamma_k(\lambda)\widehat g_k(\lambda)=1$ for all
$\lambda$, which corresponds to the exact inversion of the associated diagonal
operator in the graph Fourier basis. In all numerical experiments,
$m(\lambda)$ was observed to be bounded away from zero, yielding a stable and
well-conditioned reconstruction.

To assess the performance of the WGFT and MWGFT, test signals are generated
either directly in the vertex domain or by prescribing a spectral profile. In
particular, smooth signals are obtained by defining $\widehat f(\lambda_\ell)$
and computing
\[
f = U\,\widehat f,
\]
which ensures consistency with the graph Fourier framework.

\subsection{WGFT and MWGFT computation.}
For each window $g_j$, WGFT atoms are constructed according to
Definition~\eqref{e2} by combining generalized translation and modulation
operators. The WGFT coefficients
\[
S_f^{(j)}(n,k)=\langle f,g^{j}_{n,k}\rangle
\]
are computed for all vertices $i$ and frequency indices $k$, yielding a
vertex--frequency representation of the signal. In the multi--window case,
this procedure is repeated for all windows, and the resulting coefficients
are combined in the synthesis stage.

Signal reconstruction is performed using the multi--window synthesis formula
given in Theorem~\ref{Recteo2}. The reconstruction requires the computation
of the quantities $\langle T_n\gamma_j,T_n g_j\rangle$, which are evaluated
explicitly for all vertices and windows. Their nonvanishing is verified
numerically, in agreement with the sufficient conditions derived in
Sections~\ref{sec:WFT_rec_formula} and~\ref{sec:MWGFT}.

For clarity, Algorithm \ref{alg:MWGFT} summarizes the MWGFT analysis and synthesis procedure used in all experiments. In the implementation, translated windows are handled implicitly via the graph
filtering operators $M_{g_j}$ and $M_{\gamma_j}$, without explicitly forming all
translations $T_n g_j$.

\begin{algorithm}[H]
\caption{Multi--Windowed Graph Fourier Transform (MWGFT): Analysis and Synthesis}
\label{alg:MWGFT}
\begin{algorithmic}[1]
\Require Graph Laplacian $\mathcal{L}$, signal $f\in\mathbb{C}^N$, spectral analysis windows
$\{\widehat g_j\}_{j=1}^J$, spectral synthesis windows $\{\widehat\gamma_j\}_{j=1}^J$
\Ensure Reconstructed signal $f_{\mathrm{rec}}\in\mathbb{C}^N$

\State Compute eigendecomposition $\mathcal{L}=U\Lambda U^{*}$
\State Initialize numerator $p\gets 0\in\mathbb{C}^N$ and denominator $d\gets 0\in\mathbb{C}^N$

\For{$j=1$ to $J$}
    \State Construct graph filters
    \[
      M_{g_j}=U\,\mathrm{diag}(\widehat g_j)\,U^{*},
      \qquad
      M_{\gamma_j}=U\,\mathrm{diag}(\widehat\gamma_j)\,U^{*}
    \]
    \State Compute WGFT coefficients $S_f^{(j)}(i,k)=\langle f,g^{(j)}_{i,k}\rangle$ for all $(i,k)$
    \State Compute denominator terms $d_j(n)=\langle T_n\gamma_j,T_n g_j\rangle$ for all $n$
    \State Accumulate synthesis contributions $p \leftarrow p + \sum_{i,k} \langle f, g^{(j)}_{i,k}\rangle\,\gamma^{(j)}_{i,k}$
    \State Accumulate denominator $d \gets d + d_j$
\EndFor

\State Reconstruct
\[
  f_{\mathrm{rec}}(n)=\frac{p(n)}{N\,d(n)},\qquad n=1,\dots,N
\]
\Return $f_{\mathrm{rec}}$
\end{algorithmic}
\end{algorithm}

\subsection{Numerical validation and visualization.}
Reconstruction accuracy is quantified by the relative $\ell^2$-error between
the original and reconstructed signals. WGFT and MWGFT spectrograms are
displayed as two-dimensional vertex-frequency maps, and vertex--domain
signals are visualized on the graph using node coordinates when available.
These results illustrate the localization and stability properties of the
proposed transforms.

\subsection*{Path graph, impulse signal, RBF kernel}
We begin our numerical investigation with a simple and interpretable test case,
namely a path graph and an impulse signal.
This setting allows us to isolate the basic properties of the proposed MWGFT
framework, including exact reconstruction and vertex–frequency localization,
before moving to more complex signal and graph structures.

Figure~\ref{fig:path_signal_graph} shows the reconstruction performance of the proposed MWGFT framework on a path graph with 50 nodes, for an impulse signal supported at node 25. The original and reconstructed signals are represented on the graph vertices and coincide up to numerical precision, illustrating the exact recovery up to numerical precision guaranteed by Theorem~\ref{Recteo2}.

\begin{figure}[H]
\centering
\includegraphics[width=0.45\linewidth]{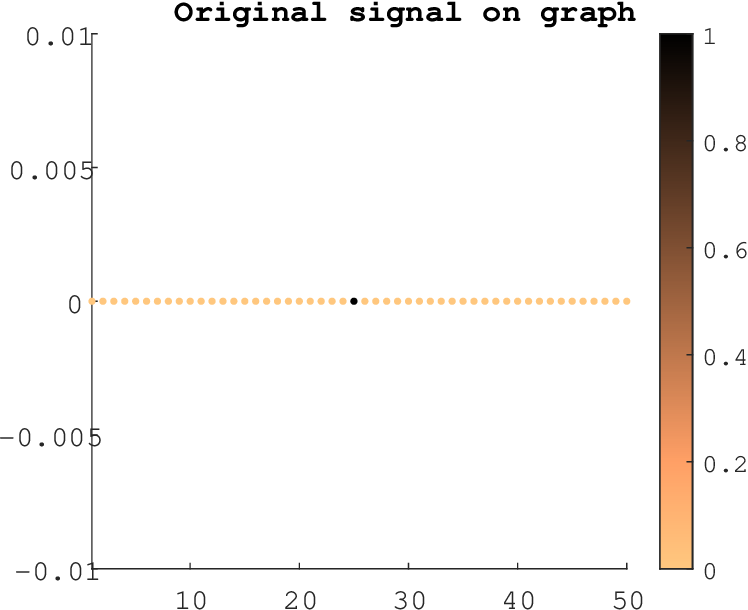}
\includegraphics[width=0.45\linewidth]{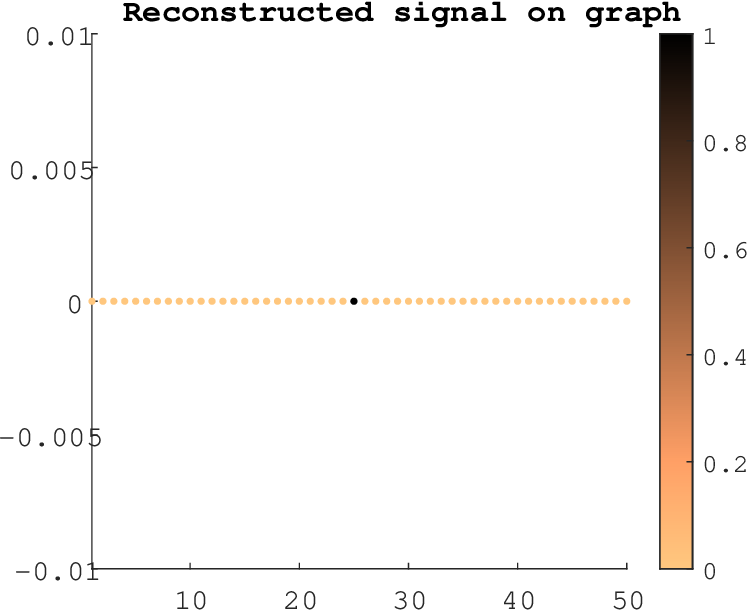}
\caption{Original and reconstructed signals on the graph. The left panel shows the original signal $f$ defined on the graph vertices, while the right panel displays the reconstructed signal $f_{rec}$ obtained via the multi-windowed graph Fourier transform synthesis formula. The close visual agreement confirms exact signal recovery up to numerical precision in the vertex domain.}
\label{fig:path_signal_graph}
\end{figure}

Figure~\ref{fig:path_graph_windows} illustrates the spectral windows employed in the MWGFT framework for both analysis and synthesis. The analysis windows $\widehat{g_j}(\lambda)$ are obtained by shifting a prototype RBF kernel across the spectrum of the normalized graph Laplacian, providing a collection of overlapping filters that explore different frequency regions. Their shapes reflect the frequency localization properties induced by the choice of the parameter $\ell_{\mathrm{fac}}$, here $\ell_{\mathrm{fac}}= 0.7$.

The corresponding synthesis windows $\widehat{\gamma_j}(\lambda)$ are constructed through normalization with respect to the spectral energy response, ensuring stable reconstruction. As shown in the figure, the synthesis windows preserve the overall spectral coverage while compensating for the overlap among analysis windows. This normalization plays a crucial role in guaranteeing well-conditioned reconstruction and allows the MWGFT to balance frequency selectivity and vertex localization.

\begin{figure}[H]
\centering
\includegraphics[width=0.45\linewidth]{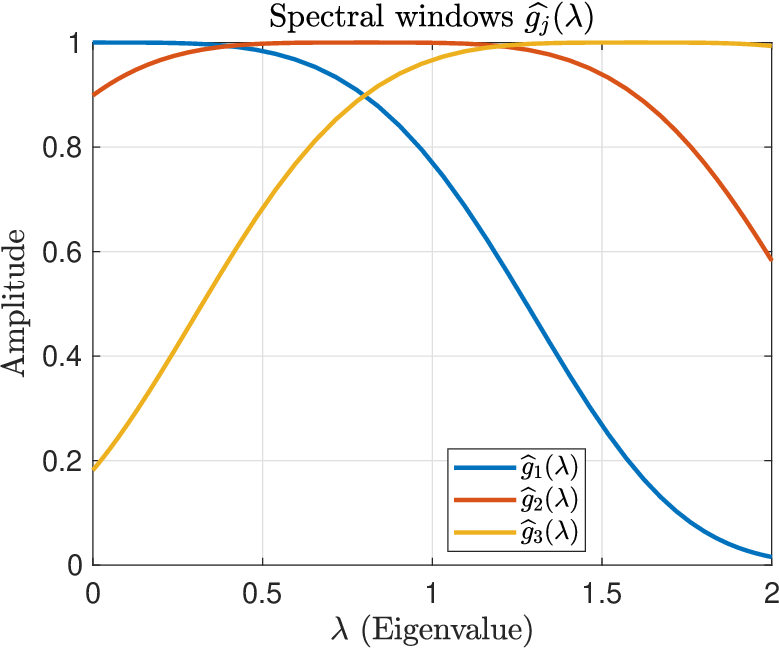}
\includegraphics[width=0.45\linewidth]{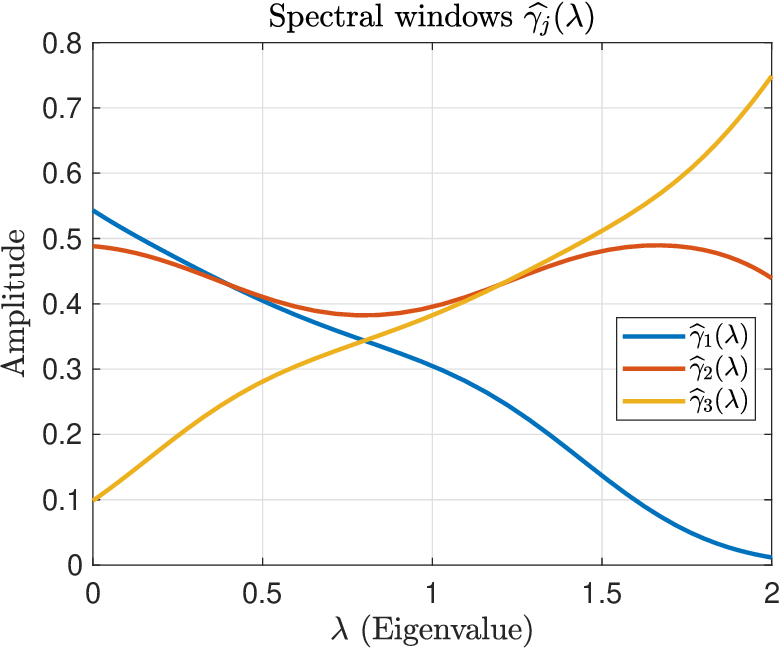}
\caption{Spectral analysis windows $\widehat{g_j}(\lambda)$ (left) and corresponding synthesis windows $\widehat{\gamma_j}(\lambda)$ (right) in the graph Fourier domain.}
\label{fig:path_graph_windows}
\end{figure}

Figure \ref{fig:path_graph_spectogram} reports the MWGFT spectrograms obtained for an impulse signal on a path graph. The averaged spectrogram provides a global representation of the energy distribution across graph vertices and graph Fourier modes, while the individual spectrograms corresponding to each window highlight how the multi-window construction decomposes the broadband impulse into components with different vertex–frequency localizations. In particular, the impulse response remains spatially concentrated around the excited vertex, while the distribution across graph Fourier modes reflects the spectral selectivity induced by the window design. We emphasize that the horizontal axis of the spectrogram indexes Laplacian
eigenmodes ordered by increasing eigenvalue, rather than the eigenvalues
themselves.

\begin{figure}[H]
\centering
\includegraphics[width=0.5\linewidth]{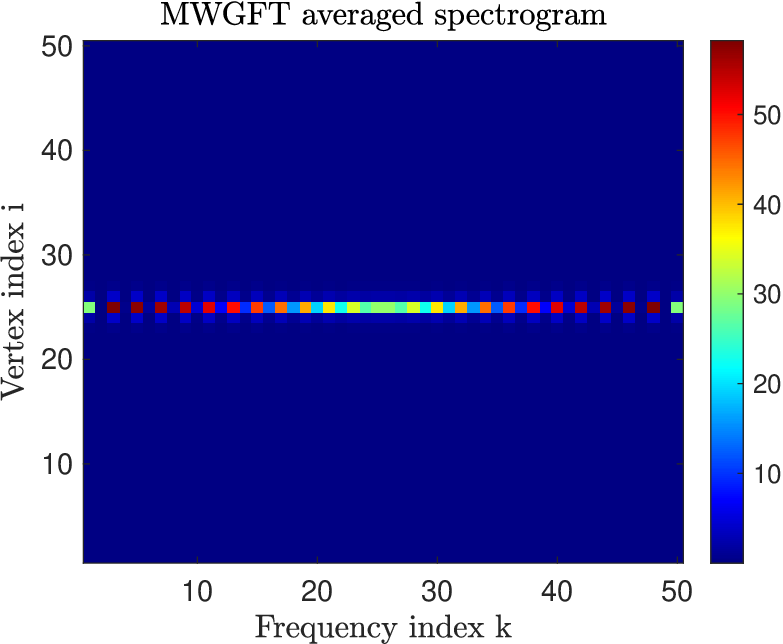} \\
\vspace{0.3cm}
\includegraphics[width=0.3\linewidth]{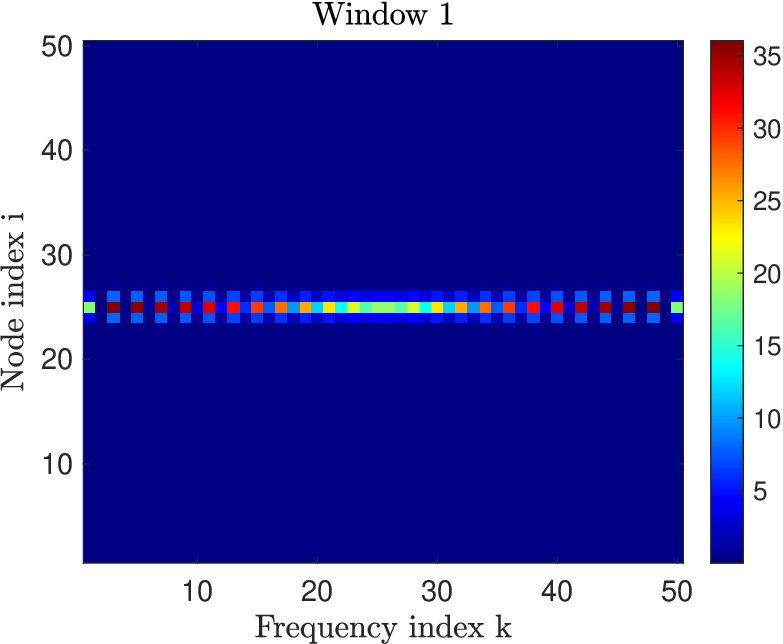}
\includegraphics[width=0.3\linewidth]{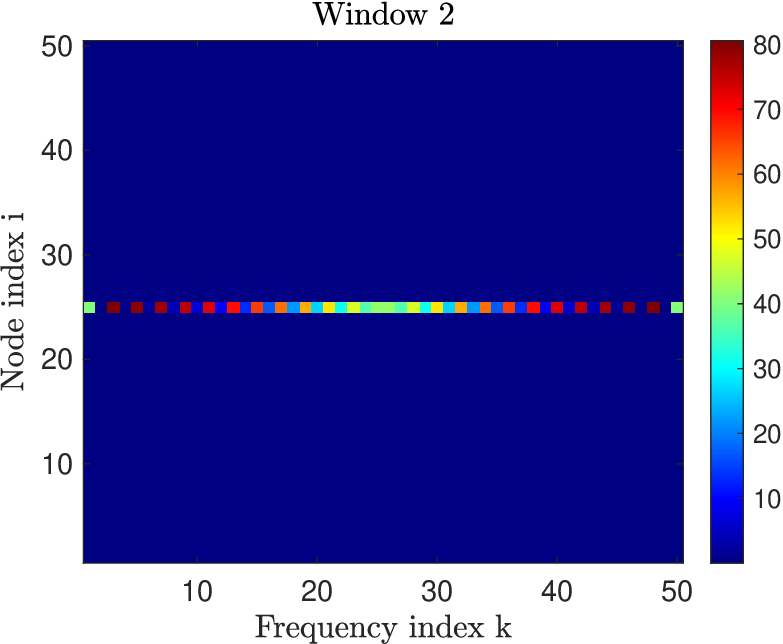}
\includegraphics[width=0.3\linewidth]{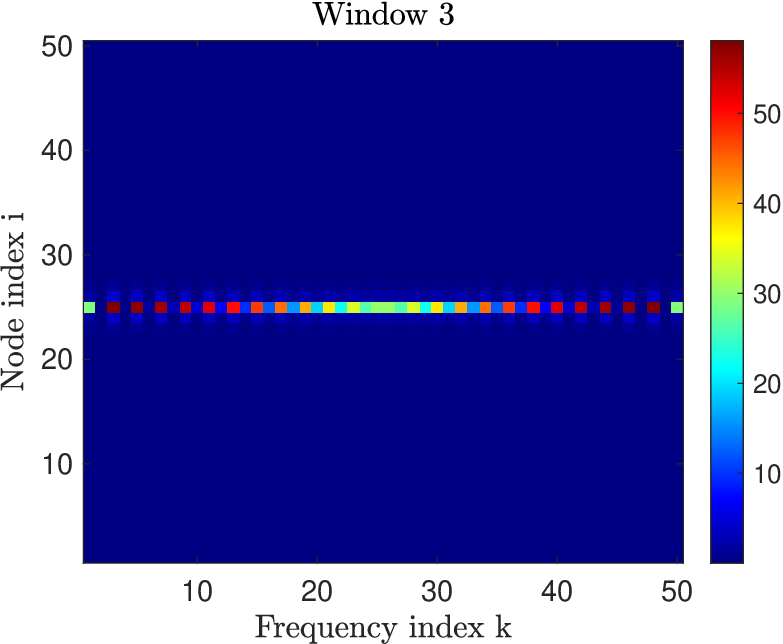}
\caption{MWGFT spectrograms of an impulse on a path graph. From left to right: averaged spectrogram over the three windows, and spectrograms corresponding to each individual window. Here, the vertical axis corresponds to graph vertices, while the horizontal axis indexes graph Fourier modes (i.e., Laplacian eigenmodes).}
\label{fig:path_graph_spectogram}
\end{figure}

\subsection*{Minnesota graph, heat signal, RBF kernel}

Here we consider a real-world irregular graph, namely the Minnesota road
network, and a smooth heat-type signal to assess the performance of the proposed
MWGFT framework beyond simple graph topologies.
The use of RBF-type spectral windows allows for a controlled decomposition of the
signal across different graph-frequency regions while preserving spatial
localization. Here, differently from the previous example, the analysis and synthesis windows are the same. This choice highlights that stability is preserved even when canonical dual
windows are not used.

As illustrated in Figure~\ref{fig:minnesota_signal_graph}, the original and reconstructed signals
are visually indistinguishable on the graph and in the vertex domain, with the
reconstruction error remaining at machine precision.
Figure~\ref{fig:minnesota_windows} further demonstrates that the selected family
of spectral windows yields reconstruction denominators that are uniformly
bounded away from zero across all vertices, thereby satisfying the
non-degeneracy condition required by Theorem~\ref{Recteo2}.
Together, these results confirm that the proposed MWGFT framework provides stable
and well-conditioned reconstruction on irregular graphs when using multiple
spectral windows, even in the presence of nonuniform graph connectivity.

Figure \ref{fig:minnesota_signal_graph} illustrates the reconstruction performance of the proposed MWGFT framework on Minnesota graph. The original signal (heat kernel) and its reconstruction are displayed directly on the graph using the vertex embedding (second row). The reconstructed signal overlaps with the original one at all vertices, while the pointwise error remains at machine precision. This result numerically validates the reconstruction formula in Theorem~\ref{Recteo2}.

\begin{figure}[H]
\centering
\includegraphics[width=0.45\linewidth]{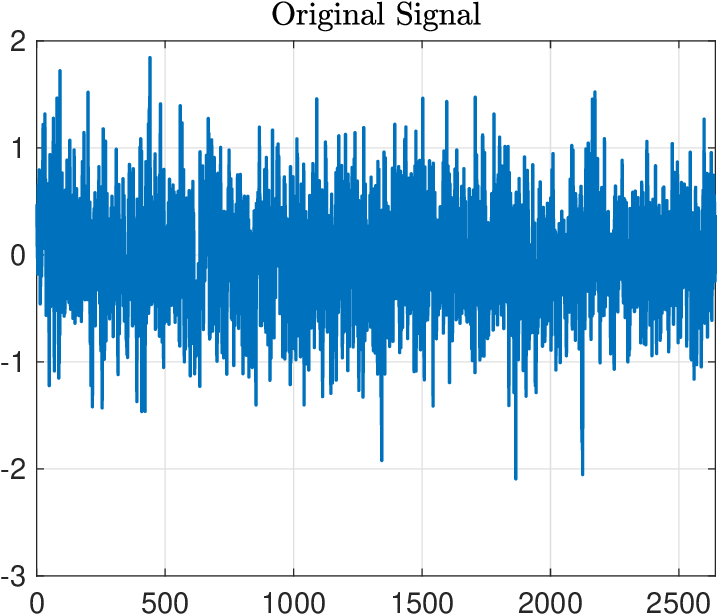}
\includegraphics[width=0.45\linewidth]{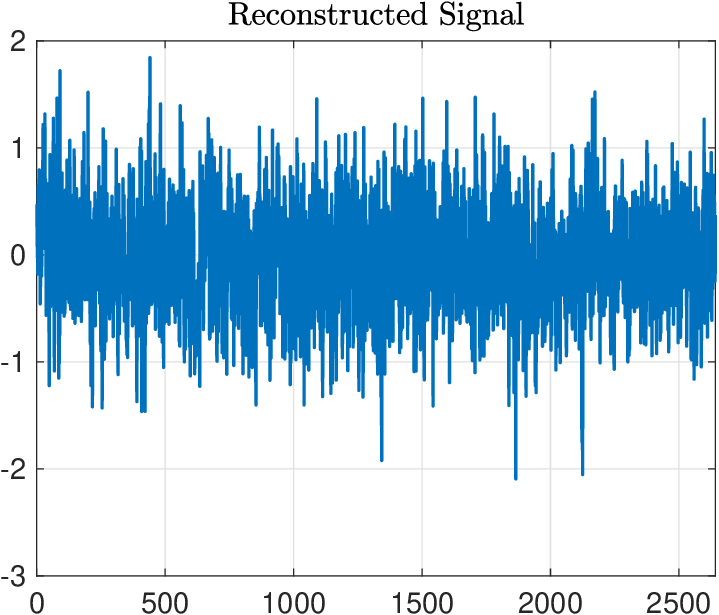}

\includegraphics[width=0.45\linewidth]{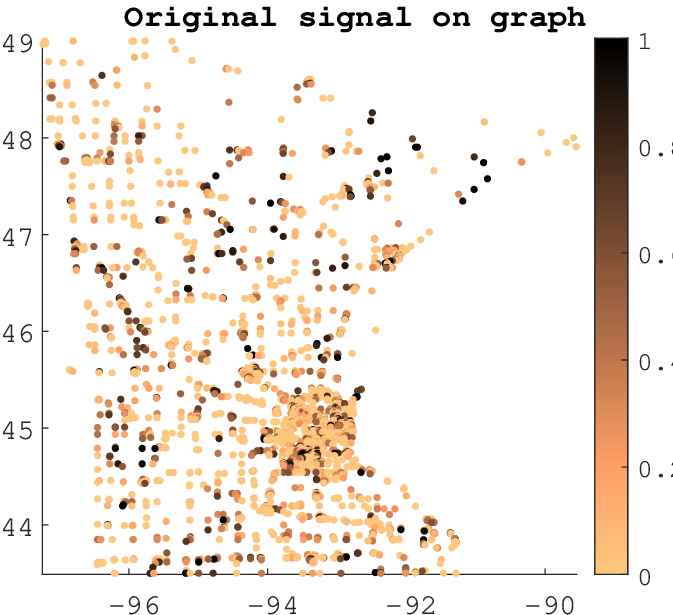}
\hspace{0.3cm}
\includegraphics[width=0.45\linewidth]{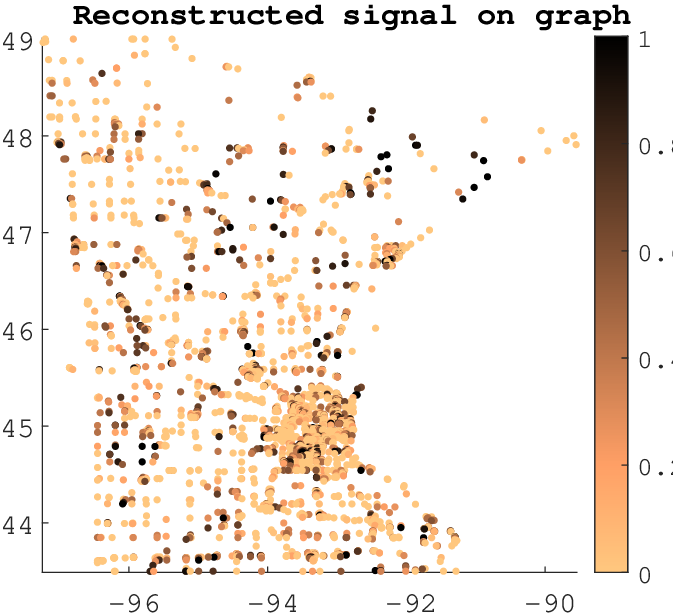}

\vspace{0.3cm}
\includegraphics[width=0.45\linewidth]{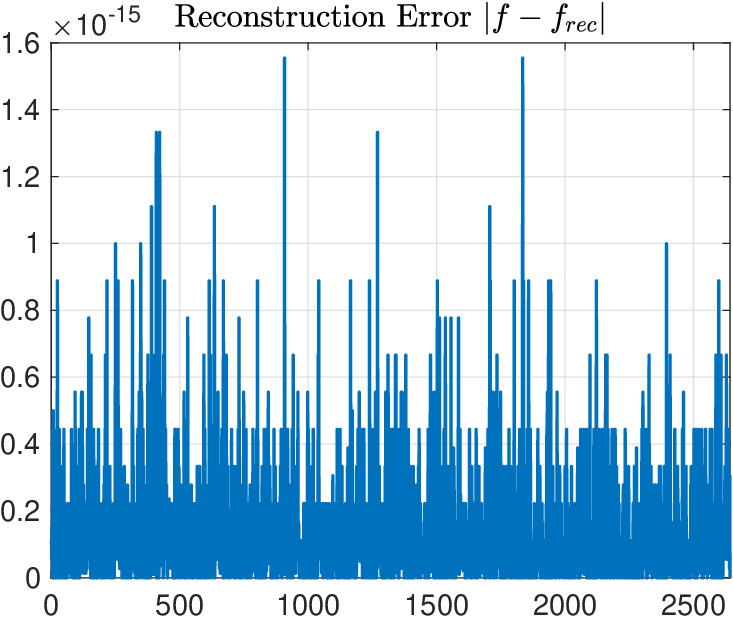}
\caption{Original and reconstructed signals on the graph. The left panel shows the original signal $f$ defined on the graph vertices, while the right panel displays the reconstructed signal $f_{rec}$ obtained via the multi-windowed graph Fourier transform synthesis formula. The close visual agreement confirms exact signal recovery in the vertex domain. Third row shows pointwise reconstruction error $|f-f_{rec}|$. The error remains at numerical precision, confirming the stability of the reconstruction.}
\label{fig:minnesota_signal_graph}
\end{figure}


Figure \ref{fig:minnesota_windows} illustrates the relationship between the choice of spectral windows and
the stability of the reconstruction formula.
The top panels report the spectral profiles of three representative RBF-type
analysis windows, corresponding to different frequency localizations.
The bottom panels show the associated reconstruction denominators
$\left|\sum_{j} \langle T_n \gamma_j,\, T_n g_j\rangle\right|$ evaluated at each
vertex.

\begin{figure}[H]
\centering
\includegraphics[width=0.3\linewidth]{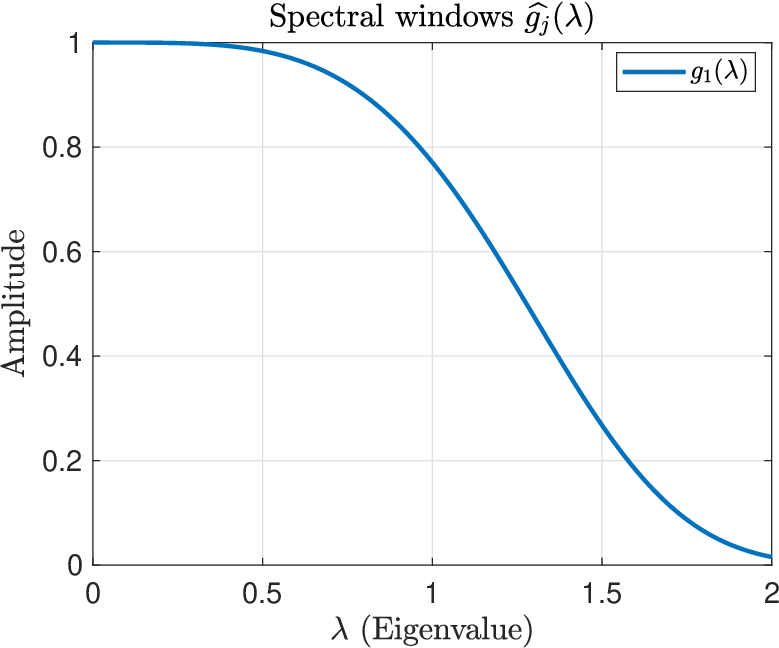}
\includegraphics[width=0.3\linewidth]{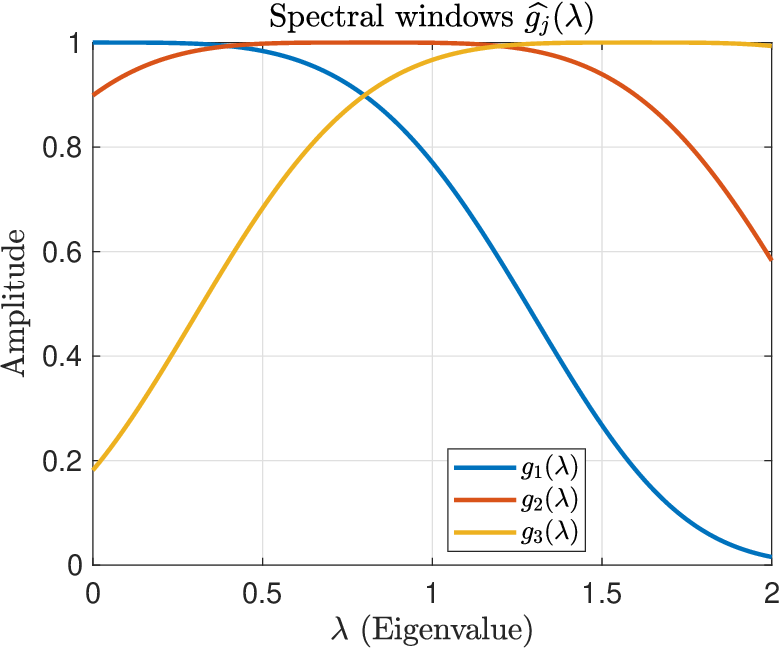}
\includegraphics[width=0.3\linewidth]{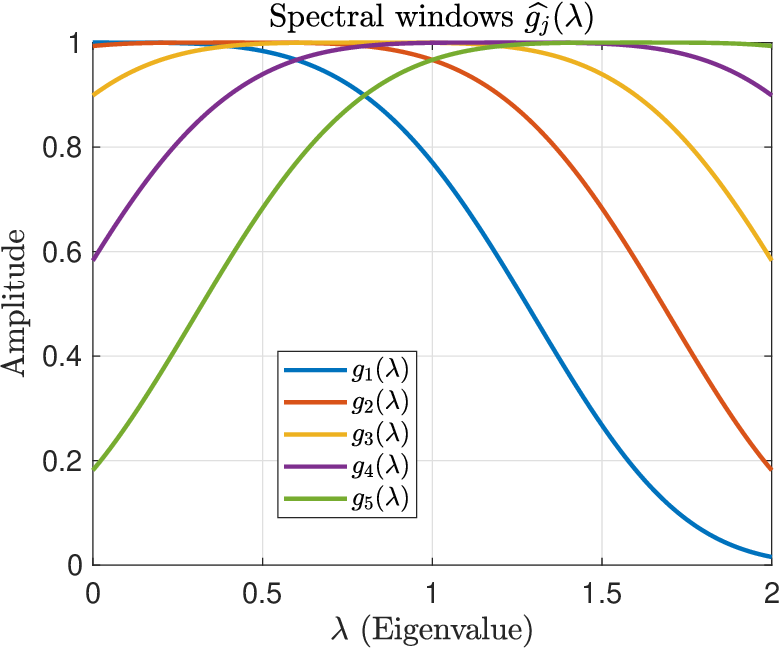}

\vspace{0.2cm}

\includegraphics[width=0.3\linewidth]{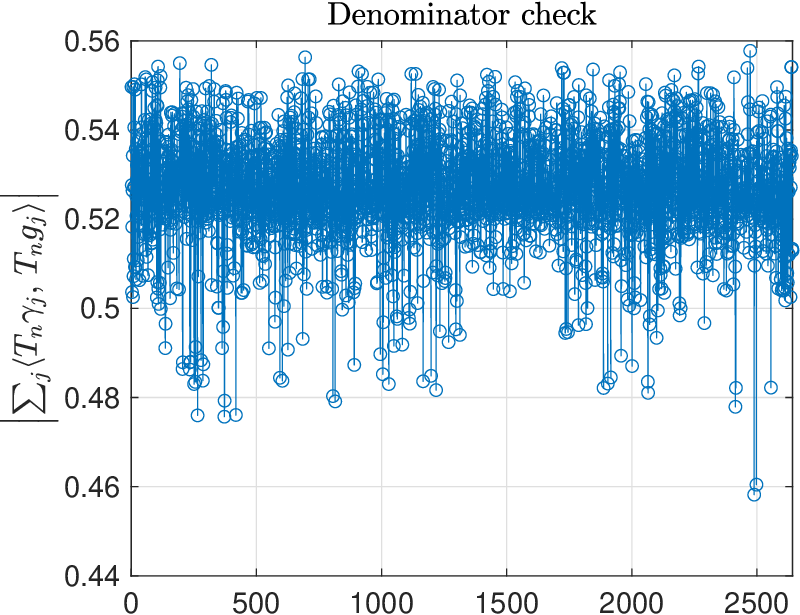}
\includegraphics[width=0.3\linewidth]{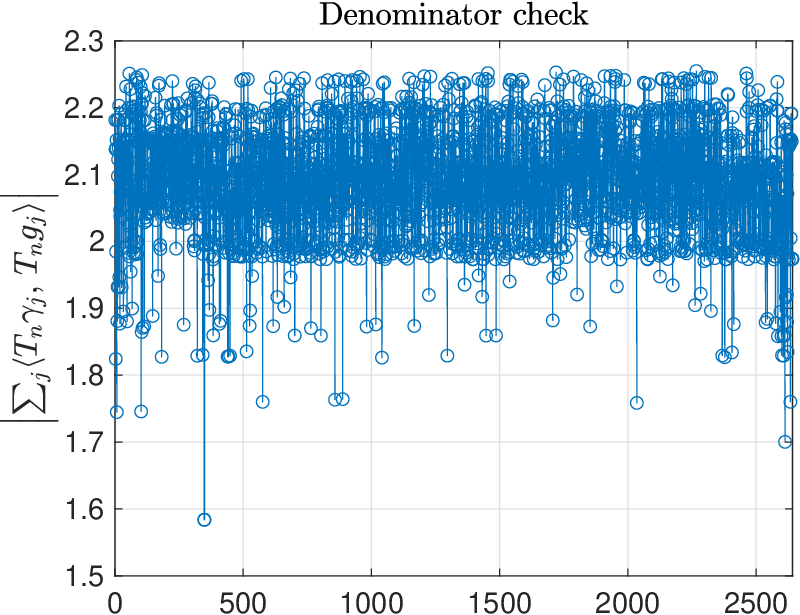}
\includegraphics[width=0.3\linewidth]{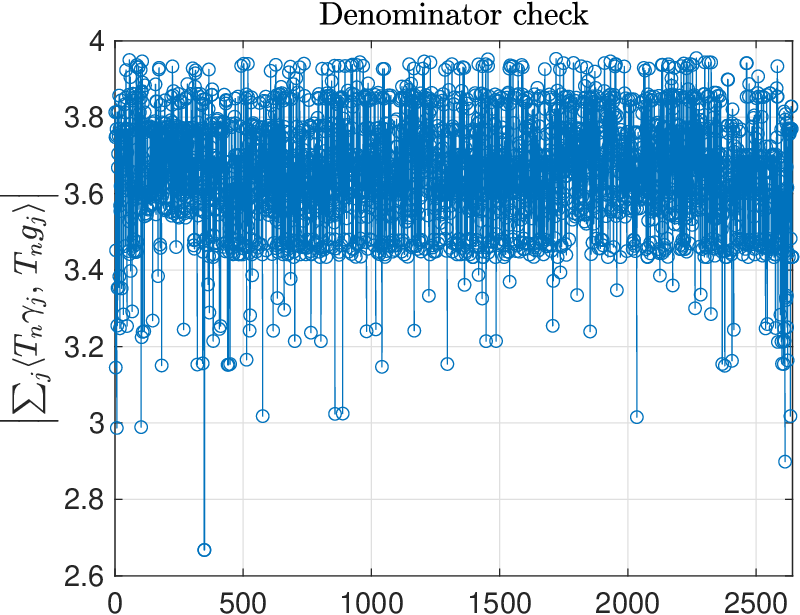}
\caption{Spectral windows and corresponding reconstruction denominators.
Top row: spectral profiles of the analysis windows $\widehat g_j(\lambda)$
(RBF kernel type) for window indices $j=1,3,5$.
Bottom row: corresponding values of the reconstruction denominator
$\left|\sum_{j} \langle T_n \gamma_j,\, T_n g_j\rangle\right|$ as a function of the
vertex index $n$. The denominator remains strictly positive for all vertices, ensuring stable
reconstruction.}
\label{fig:minnesota_windows}
\end{figure}
For all considered window selections, the denominator remains uniformly bounded
away from zero across the graph, thereby satisfying the non-degeneracy condition
required by Theorem~\ref{Recteo2}.
These results confirm that the chosen window designs lead to stable and
well-conditioned reconstructions even when reconstruction is performed using synthesis windows that differ from the canonical duals and relies on the corresponding normalization.

\subsection*{Path graph, chirp vertex signal, RBF kernel}
We finally consider a path graph with $50$ nodes and a vertex-localized chirp signal,
defined as a Gaussian envelope modulated by a linearly varying phase, Figure~\ref{fig:path_chirp_signal_vertex_graph}.

\begin{figure}[H]
\centering
\includegraphics[width=0.45\linewidth]{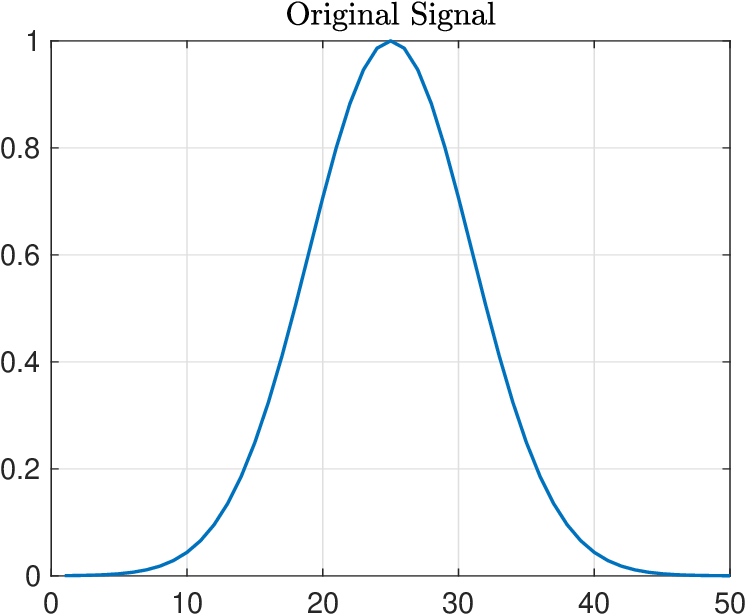}
\includegraphics[width=0.45\linewidth]{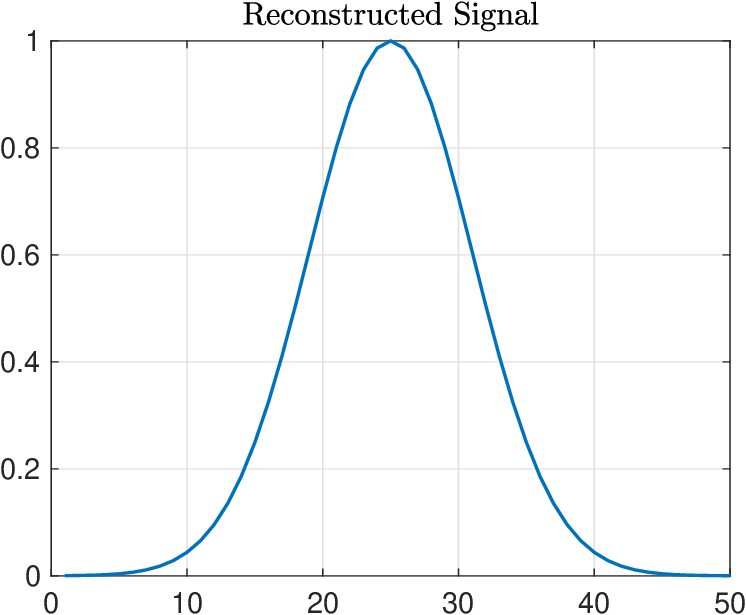}

\vspace{0.3cm}
\includegraphics[width=0.45\linewidth]{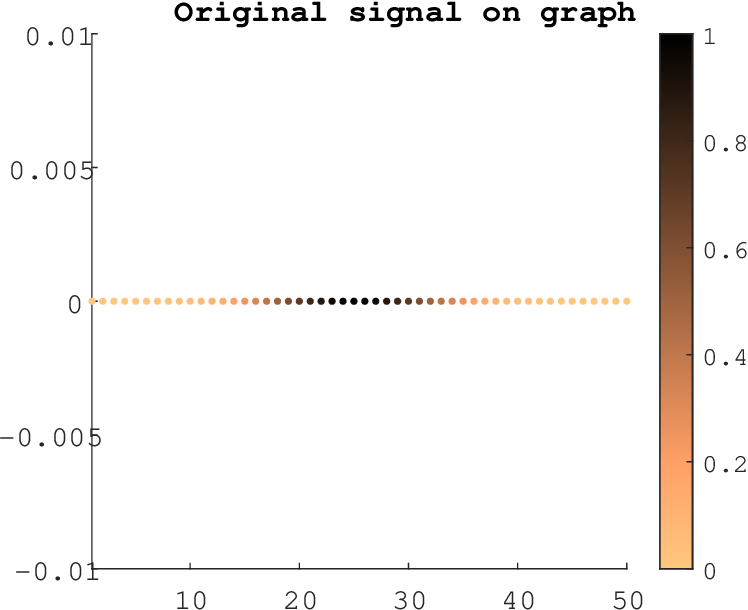}
\includegraphics[width=0.45\linewidth]{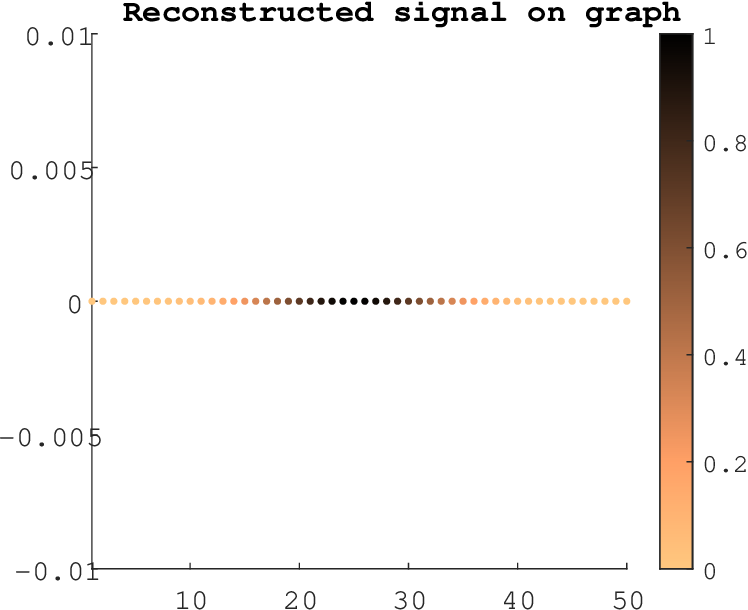}
\caption{Top: vertex-domain representation of the reconstruction process. Bottom: signal defined on the graph vertices. From left to right: original signal $f$, reconstructed signal $f_{rec}$.}
\label{fig:path_chirp_signal_vertex_graph}
\end{figure}

A vertex-localized chirp signal is used to demonstrate joint vertex–frequency
localization and is defined by
\[
f(n)
=
\exp\!\left(-\frac{(n-i_0)^2}{2\sigma^2}\right)
\exp\!\big(i\,\alpha (n-i_0)\big),
\qquad n=1,\dots,N,
\]
where the Gaussian envelope ensures spatial localization around the vertex $i_0$,
while the complex exponential introduces oscillatory behavior across the graph.
The parameter $\sigma = 6$ controls the spatial extent of the signal and $\alpha = 0.3$
determines its frequency content. Although the phase varies linearly with the vertex index, this signal already
produces a structured ridge in the vertex–frequency plane, sufficient to
illustrate joint localization effects. This construction yields a compact and
interpretable test signal exhibiting both vertex localization and nontrivial
graph-frequency structure.

This signal exhibits joint localization in the vertex domain and in graph
frequencies, making it well suited to illustrate the capabilities of the proposed
MWGFT framework.
The analysis is performed using $6$ RBF-type spectral windows, $\ell_{\mathrm{fac}}= 0.5$, uniformly shifted
across the spectrum of the normalized graph Laplacian, Figure~\ref{fig:path_chirp_graph_windows} (left), and for the reconstruction six synthesis windows defined as in \eqref{eqn:gamma_synthesis}, Figure~\ref{fig:path_chirp_graph_windows} (right).

\begin{figure}[H]
\centering
\includegraphics[width=0.45\linewidth]{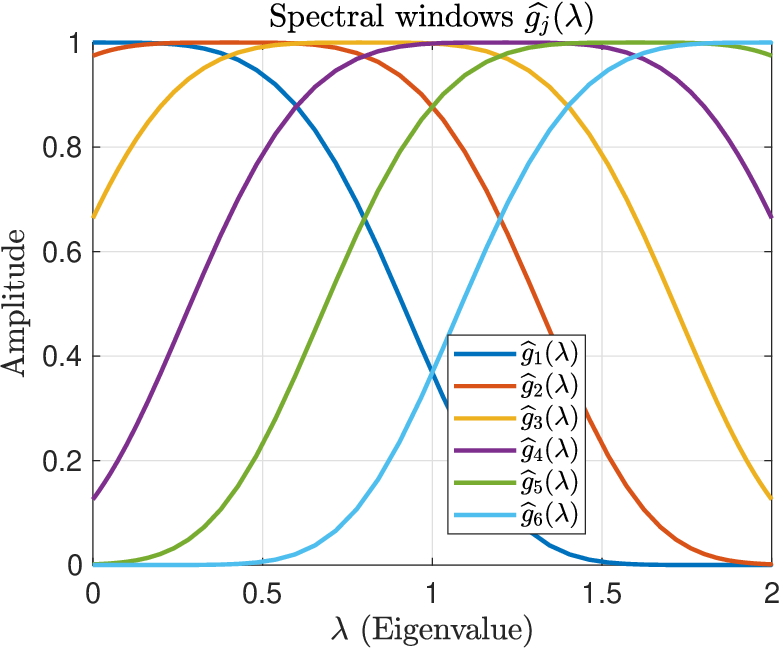}
\includegraphics[width=0.45\linewidth]{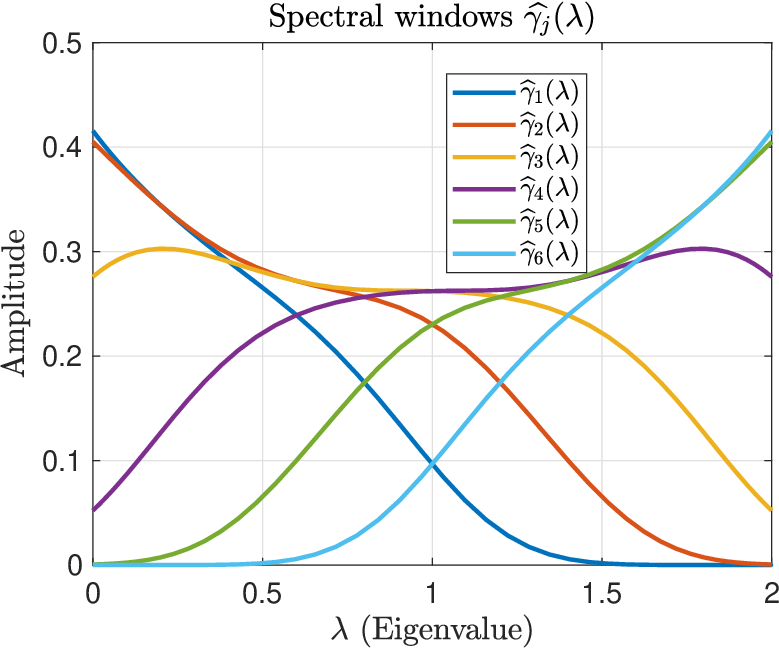}
\caption{Spectral analysis windows ${g_j}(\lambda)$ (left) and corresponding synthesis windows ${\gamma_j}(\lambda)$ (right) in the graph Fourier domain.}
\label{fig:path_chirp_graph_windows}
\end{figure}

As shown in Figure~\ref{fig:path_chirp_graph_spectogram}, MWGFT spectrograms reveal a compact spatial support around the
center vertex together with a structured distribution across graph Fourier modes.
In particular, the averaged spectrogram highlights a ridge-like pattern in the
vertex–frequency plane, which reflects the chirp behavior of the signal and
demonstrates effective joint vertex–frequency localization.

\begin{figure}[H]
\centering
\includegraphics[width=0.5\linewidth]{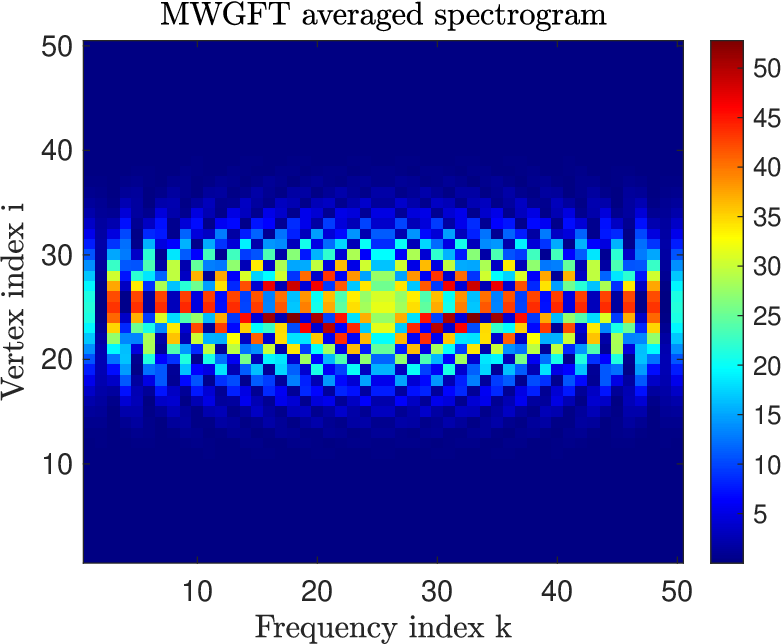} \\
\vspace{0.2cm}

\includegraphics[width=0.3\linewidth]{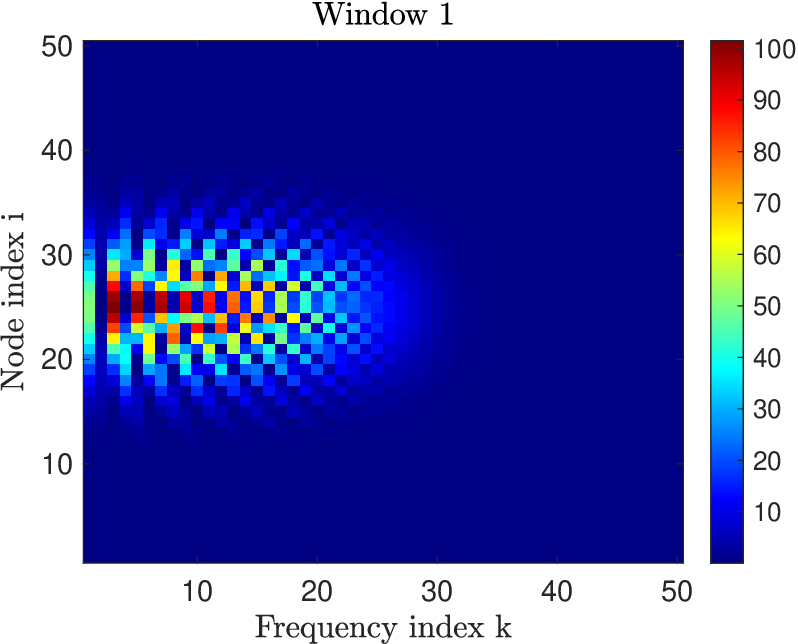}
\includegraphics[width=0.3\linewidth]{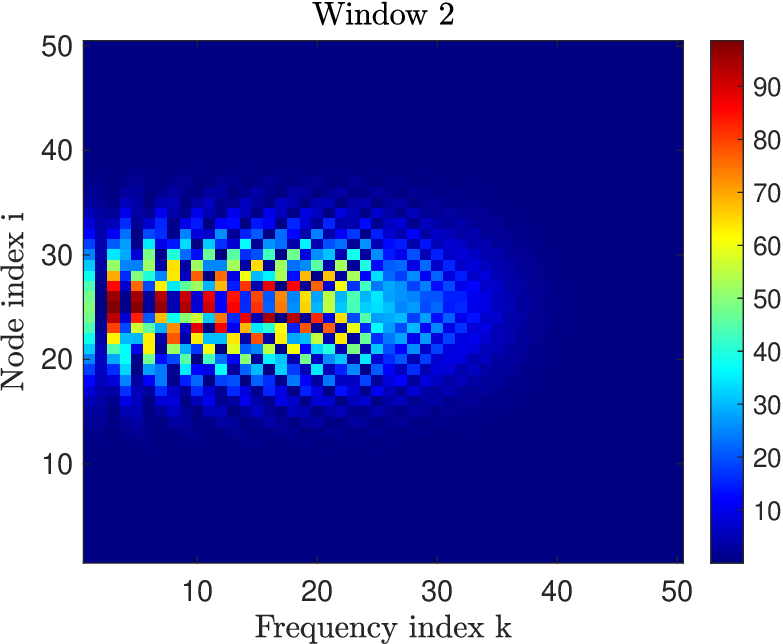}
\includegraphics[width=0.3\linewidth]{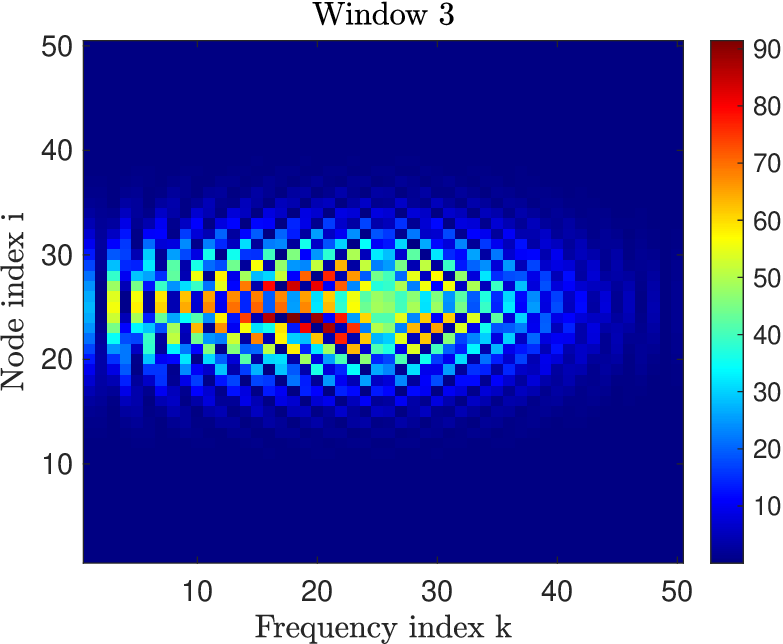}

\includegraphics[width=0.3\linewidth]{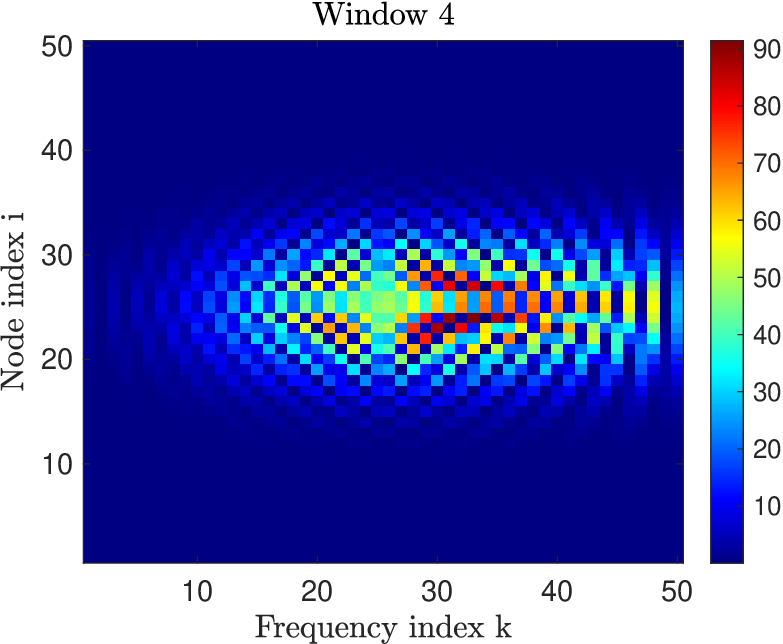}
\includegraphics[width=0.3\linewidth]{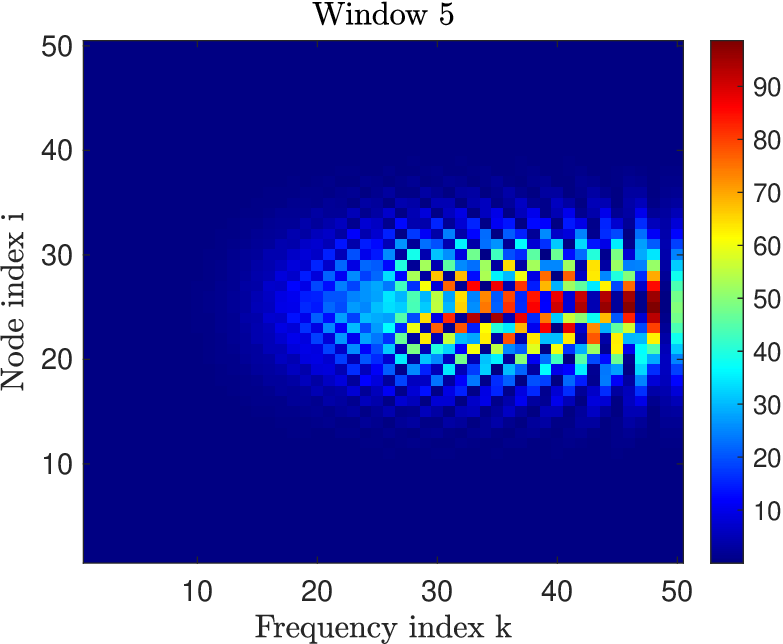}
\includegraphics[width=0.3\linewidth]{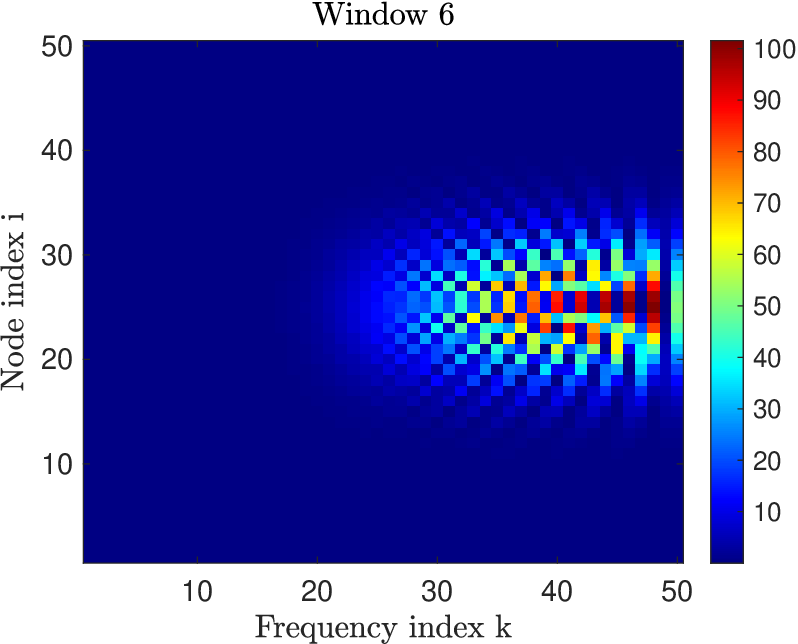}

\caption{MWGFT spectrograms of the chirp-vertex-signal on a path graph. Top: averaged spectrogram over the six windows. Bottom: spectrograms corresponding to each individual window. Here, the vertical axis corresponds to graph vertices, while the horizontal axis indexes graph Fourier modes (i.e., Laplacian eigenmodes).}
\label{fig:path_chirp_graph_spectogram}
\end{figure}

For convenience, Table~\ref{tab:numerical_summary} summarizes the numerical
experiments presented in this section, together with their main objectives.

\begin{table}[H]
\centering
\caption{Summary of numerical experiments presented in Section~5.}
\label{tab:numerical_summary}
\begin{tabular}{llllp{4.3cm}}
\hline
\textbf{Graph} & \textbf{Signal} & \textbf{Windows} & \textbf{Figures} & \textbf{Purpose} \\
\hline
Path ($N=50$)
& Impulse
& RBF (3 windows)
& Figs.~1--3
& Exact reconstruction; basic vertex--frequency behavior \\

Minnesota
& Heat signal
& RBF (5 windows)
& Figs.~4--5
& Stability on irregular graphs; denominator verification \\

Path ($N=50$)
& Vertex chirp
& RBF (6 windows)
& Figs.~6--8
& Joint vertex--frequency localization; MWGFT advantage \\
\hline
\end{tabular}
\end{table}

\section{Conclusions}\label{sec:conclusions}

In this work we introduced a multi-windowed graph Fourier transform (MWGFT) for
the analysis of signals defined on finite graphs. The proposed framework extends
the windowed graph Fourier transform by allowing multiple analysis and synthesis
windows, thereby providing increased flexibility in the joint vertex–frequency
representation of graph signals. Exact reconstruction formulas were derived for
complex-valued signals, together with sufficient conditions ensuring stable
invertibility. These results naturally lead to frame representations associated
with families of generalized translated and modulated windows.

The numerical experiments confirmed the theoretical findings and illustrated
the practical relevance of the MWGFT. In particular, exact reconstruction was
observed up to machine precision in all considered cases, and the non-degeneracy
conditions required by the reconstruction formulas were verified numerically.
Experiments on both regular and irregular graphs highlighted the ability of the
multi-window approach to improve robustness and interpretability compared to
single-window constructions, especially in terms of vertex–frequency
localization.

The present implementation relies on the full eigendecomposition of the graph
Laplacian and is therefore suitable for graphs of moderate size. Extensions to
large-scale graphs based on approximate spectral filtering techniques are
currently under investigation. More generally, the proposed MWGFT provides a
flexible and theoretically grounded tool for graph signal processing, with
potential applications in the analysis of nonstationary or spatially localized
data on networks.

\section*{Acknowledgements}
IMB and SS have been supported by the project TIGRECO funded by the MUR Progetti di Ricerca di Rilevante Interesse Nazionale (PRIN) Bando 2022, Grant 20227TRY8H, (CUP-J53D23003630001 for IMB, CUP-C53D23002400006 for SS). This research has been accomplished within RITA (Research ITalian network on Approximation) and the UMI Group TAA (Approximation Theory and Applications). \par
The last three authors are members of Gruppo Nazionale per l’Analisi Matematica, la Probabilit\`a e le loro Applicazioni (GNAMPA) --- Istituto Nazionale di Alta Matematica (INdAM).

\bibliographystyle{abbrv}
\bibliography{Bib_WGFT}   

\end{document}